\documentclass{article}

\usepackage{geometry}
\usepackage{amssymb}
\usepackage{amsthm}
\usepackage{amsmath}
\usepackage{algorithm}
\usepackage{algpseudocode}
\usepackage{xcolor}
\usepackage{graphicx,caption,subcaption}
\usepackage{multirow, multicol}
\usepackage{authblk}

\newtheorem*{remark}{Remark}

\newtheorem{prop}{Proposition}
\theoremstyle{definition} 
\newtheorem{definition}{Definition}[section] 
\newtheorem{corollary}{Corollary}

\newcommand{\Enl}[1]{E_{\mathrm{nl}}(#1)}
\newcommand{\mob}{\mathcal M}
\newcommand{\dis}{\mathcal D}
\newcommand{\Lop}{\mathcal L}

\newcommand{\eps}{\varepsilon}
\newcommand{\dt}{\Delta t}
\newcommand{\Esym}{E_{\mathrm{sym}}}

\title{Energy-rate equidistribution for adaptive SSAV--BDF2
  integration of gradient flows}
\author{Suki Shergill \thanks{Email: sks78@sfu.ca} \and JF Williams\thanks{Email: jfwillia@sfu.ca}}
\affil{Department of Mathematics, Simon Fraser University}
\date{\today}

\begin{document}

\maketitle

\begin{abstract}
We develop an adaptive time-stepping strategy for scalar auxiliary variable (SAV) schemes combined with variable-step BDF2 discretization for gradient flows. The strategy selects step sizes based on the instantaneous energy dissipation rate and enforces monotonicity of a computable symmetric energy through a step rejection mechanism.

We show that the step-size rules used in practice correspond to equidistributing a monitor function along the energy curve, connecting energy-adaptive time stepping to the classical theory of r-adaptive mesh methods. The monitor $g = |\dot{E}|$ is asymptotically optimal as it minimises the leading-order maximum symmetric-energy discrepancy for a given number of steps. A combined monitor that switches to $g = \sqrt{|\dot{E}|}$ near equilibrium caps step growth without degrading this error bound.  The arclength parameter $\beta$ used in prior work is shown to be the ratio of the largest permitted step to the energy-error tolerance.

The controller is paired with the SSAV--BDF2 discretization. Accepted steps satisfy monotonicity of the symmetric energy, from which we deduce a bound on the physical energy that is uniform in the number of steps, $\sup_n E(u^n) \le E(u^{0}) + \tau_1 + O(\Delta t_{\max}^2)$, and the scheme achieves second-order global accuracy under standard step-ratio restrictions. Numerical experiments on Cahn--Hilliard, Ohta--Kawasaki, Phase--Field Crystal, and functionalized Cahn--Hilliard models confirm the theory and demonstrate the method's efficiency.

\end{abstract}

\section{Introduction}

Gradient flows of free-energy functionals model many pattern-forming
systems in materials science, fluid mechanics, and biology.
Prototypical examples include the Cahn--Hilliard (CH) equation for phase
separation in binary alloys \cite{CahnHilliard}, the Ohta--Kawasaki (OK)
model for block-copolymer morphology \cite{OhtaKawasaki1986}, the
Phase--Field Crystal (PFC) model describing crystalline order on
diffusive time scales \cite{Elder2002PFC,Elder2004PFC}, and the
Functionalized Cahn--Hilliard (FCH) equations for amphiphilic
interfaces \cite{Gompper1990FCH}.

All these models share a common mathematical structure. Let $H$ be a real Hilbert space with inner product
$\langle\cdot,\cdot\rangle$. We consider energies of the abstract form
\[
E(u)=\tfrac12\langle u,\mathcal L u\rangle+E_{\mathrm{nl}}(u),
\]
where $\mathcal L:D(\mathcal L)\subset H\to H$ is a densely defined
self-adjoint differential operator that is bounded below on the relevant
state space, and $E_{\mathrm{nl}}$ is Fr\'echet differentiable. We write
\[
G(u):=\frac{\delta E_{\mathrm{nl}}}{\delta u}.
\]
For each model the governing gradient flow takes the form
\[
u_t \;=\; \mob\,\frac{\delta E(u)}{\delta u},
\qquad
\frac{\delta E}{\delta u}=\mathcal L u+G(u),
\]
where $\mob$ is a self-adjoint non-positive mobility operator.
In the sequel we take $\mob=\Delta$ with periodic (or homogeneous
Neumann) boundary conditions, so that
\[
u_t = \Delta(\mathcal{L}u+G(u)).
\]
$E_{\mathrm{nl}}(u)=\int_\Omega F(u)\,dx$, one has $G(u)=F'(u)$.
This structure guarantees \emph{energy dissipation}.
Defining the chemical potential
$\mu=\delta E/\delta u=\mathcal{L}u+G(u)$, we have
\[
\frac{dE}{dt}
=\big\langle \mu,u_t\big\rangle
=\big\langle \mu,\Delta\mu\big\rangle
=-\,\|\nabla\mu\|_{L^2}^2\le0,
\]
which expresses the physical irreversibility of relaxation toward
critical points of the energy.
The same structure enforces a mass constraint,
$\frac{d}{dt}\int_\Omega u(x,t)\,dx = 0$.

\paragraph{Energy structure and numerical challenges.}
Preserving this gradient-flow structure at the discrete level is
critical.
If a scheme fails to maintain discrete energy decay, numerical solutions
can exhibit spurious oscillations, drift in total energy, or violate
qualitative features such as coarsening rates and phase-volume
conservation \cite{WiseWangLowengrub2009}.
These equations are also stiff: the linear diffusive part imposes severe
time-step restrictions \cite{hairer1996solving}, while the nonlinear
term $G(u)$ can vary rapidly as interfaces sharpen.
Fully explicit schemes are therefore inefficient, and implicit methods
require careful design to remain energy-stable and second-order accurate.

\paragraph{SAV and BDF2 frameworks.}
The scalar auxiliary variable (SAV) method of Shen, Xu, and Yang
\cite{shen2018scalar,shen2019new} addresses this by augmenting the
system with a scalar $w(t)=\sqrt{\Enl{u(t)}+C}$, producing a modified
energy that is quadratic in $(u,w)$ and admits stable, linearly implicit
time discretizations.

SAV--BDF2 and stabilized SAV--BDF2 (SSAV--BDF2) methods achieve second-order accuracy and
strong energy-dissipation properties without solving nonlinear systems at
each step \cite{shen2018scalar}.
For problems whose energy spans several time scales---coarsening,
nucleation, defect dynamics---adaptive time stepping is needed to make
these schemes efficient.

\paragraph{Motivation for adaptive, structure-preserving methods.}
Energy decay provides a physically meaningful measure of local temporal
activity.
Where evolution is slow (small dissipation rate), one can take large
steps; where it is fast (interface reconnection, defect annihilation),
finer resolution is needed.
An adaptive controller driven by the instantaneous energy dissipation
rate,
\[
\dis(t)=\int_\Omega |\nabla\mu|^2\,dx,
\]
adjusts the step size directly from the physical relaxation dynamics.

The difficulty is that variable-step BDF2 loses the telescoping energy
identity that makes constant-step analysis work, and the SAV/SSAV
coupling introduces additional consistency errors
\cite{hairer1987solving}.

\paragraph{Contributions and organization.}
We develop an energy-rate equidistribution framework for adaptive
SSAV--BDF2 integration of gradient-flow energies
\[
E(u)=\tfrac12\big\langle u,\mathcal{L}u\big\rangle
+E_{\mathrm{nl}}(u)
\]
and the $H^{-1}$ mass-preserving gradient flow
\begin{equation}\label{eq:gradient_flow}
u_t = \Delta \left(\mathcal{L} u + G(u)\right) \; .
\end{equation}

The paper makes four contributions. First, in Section~\ref{sec:equidistribution}, we show that step-size selection based on the energy dissipation rate corresponds to equidistributing a monitor function along the energy curve, and we prove that the monitor $g = |\dot{E}|$ asymptotically minimises the leading-order maximum symmetric-energy discrepancy for a fixed number of steps (Proposition~\ref{prop:equidist-optimal}).  We also give the arclength parameter $\beta$ from \cite{Qiao2011,Zhang2012} a careful interpretation in terms of the energy-error tolerance and the maximum admissible step (Proposition~\ref{prop:arclength}). This part of the analysis is independent of the time integrator and of the SAV reformulation: it applies to any one-step or multistep discretization of any gradient flow, and we introduce it with SSAV--vBDF2 only in Section~\ref{sec:algorithm}.

Second, we propose a combined monitor (Section~\ref{sec:equidistribution}) that uses $g = |\dot{E}|$ during rapid dissipation and $g = \sqrt{|\dot{E}|}$ near equilibrium, capping step growth where the first rule becomes too loose, and we show that the combined rule bounds the symmetric-energy discrepancy by the user defined parameter $\tau_1/2$ in \emph{both} regimes (Corollary~\ref{cor:uniform-local}).

Third, we pair this controller with SSAV--vBDF2 discretization and a symmetric-energy acceptance criterion (Section~\ref{sec:algorithm}). Accepted steps satisfy $\Esym^{n+1} \le \Esym^n$ by construction, from which we deduce a bound on the physical energy that is uniform in $n$ (Proposition~\ref{prop:uniform-physical}) together with a per-step bound whose deviation from monotonicity is proportional to $|\rho_n - 1|$ (Proposition~\ref{prop:physical-quasi-mono}). The scheme is second-order accurate under standard step-ratio restrictions (Proposition~\ref{prop:global-accuracy}).

Fourth, in Section~\ref{sec:examples}, we test the method on Cahn--Hilliard, Ohta--Kawasaki, Phase--Field Crystal, and functionalized Cahn--Hilliard problems. Here we 
compare efficiency against fixed-step SSAV--BDF2 at matched accuracy.

\section{Problem Setup and Discretization}

\subsection{Continuous Problem}

Let $H$ be a real Hilbert space with inner product
$\langle\cdot,\cdot\rangle$. We consider equation
\eqref{eq:gradient_flow}, where
$\Lop:D(\Lop)\subset H\to H$ is a densely defined self-adjoint operator,
bounded below on the relevant state space, with periodic or homogeneous
Neumann boundary conditions where appropriate. We split the energy into
linear ($E_{\mathrm{l}}$) and nonlinear ($E_{\mathrm{nl}}$) parts
\[
E(u) = E_{\mathrm{l}}(u) + E_{\mathrm{nl}}(u),
\qquad
E_{\mathrm{l}}(u) = \tfrac12 \big\langle u,\Lop u\big\rangle.
\]
We assume $E_{\mathrm{nl}}$ is Fr\'echet differentiable and write
\[
G(u):=\frac{\delta E_{\mathrm{nl}}}{\delta u}.
\]
For the pointwise case
\[
E_{\mathrm{nl}}(u) = \int_\Omega F(u)\,dx,
\]
one has $G(u)=F'(u)$. In this pointwise case we assume that the
nonlinear term $F: \mathbb{R} \to \mathbb{R}$ satisfies:

\begin{itemize}
\item $F \in C^3$ with bounded derivatives: $\|F''\|_{L^\infty} \leq C_F$ and $\|F'''\|_{L^\infty} \leq C_F$.
\item $F$ is bounded below: there exists $C > 0$ such that $F(u) + C \geq c_0 > 0$ for all relevant $u$.
\end{itemize}

The double-well potential $F(u) = \tfrac14(u^2-1)^2$ satisfies these conditions on any bounded subset of $\mathbb{R}$, with $C = 1/4$. Since solutions of the Cahn--Hilliard equation satisfy a maximum principle at the continuous level, and discrete solutions remain bounded under standard CFL-type or energy-stability conditions, the boundedness assumptions on $F''$ and $F'''$ hold on the relevant solution trajectories. We adopt this convention throughout: the derivative bounds on $F$ apply on the range of the computed solution, not globally.

For models with gradient-dependent nonlinearities, such as the functionalized Cahn--Hilliard energy in Section~\ref{sec:examples}, the nonlinear energy need not take the pointwise form $\int_\Omega F(u)\,dx$. In such cases we assign all gradient-dependent nonlinear terms to $\Enl{u}$, verify that $\Enl{u}+C>0$ holds for the computed solutions, and use the full variational derivative
\[
G(u)=\frac{\delta E_{\mathrm{nl}}}{\delta u}
\]
in place of the pointwise derivative $F'(u)$. The SAV construction is unchanged at the functional level, though the splitting and the variational derivative are problem-specific.

Defining $\mu=\Lop u+G(u)$ and using \eqref{eq:gradient_flow}, the energy dissipation property takes the form
\[
\frac{dE}{dt}
= \big\langle \mu,u_t\big\rangle
= \big\langle \mu,\Delta\mu\big\rangle
= -\|\nabla\mu\|_{L^2}^2 \le 0.
\]

\subsection{SAV Reformulation}
Following \cite{shen2018scalar}, we introduce the auxiliary variable
\[
w(t) = \sqrt{\Enl{u(t)} + C},
\]
which satisfies, by the chain rule for the functional $E_{\mathrm{nl}}$,
\[
\frac{dw}{dt} = \frac{1}{2\sqrt{\Enl{u} + C}}
\left\langle G(u), \frac{du}{dt}\right\rangle.
\]
The SAV system to be solved is
\begin{align}
\frac{du}{dt} &= \Delta\left(\mathcal{L}u + \frac{w}{\sqrt{\Enl{u} + C}} G(u)\right) =: \Delta \tilde\mu, \notag\\
\frac{dw}{dt} &= \frac{1}{2\sqrt{\Enl{u} + C}} \left\langle G(u), \frac{du}{dt}\right\rangle, \label{eq:sav_r}
\end{align}
where
\[
\tilde\mu = \mathcal{L}u + \frac{w}{\sqrt{\Enl{u}+C}} G(u).
\]

Define the SAV energy
\[
E_{\mathrm{SAV}}(u,w):=\tfrac12\langle u,\mathcal L u\rangle+w^2-C .
\]
Then
\begin{align*}
\frac{d}{dt}E_{\mathrm{SAV}}
&=\langle u_t,\mathcal L u\rangle+2w\,w_t \\
&=\langle u_t,\mathcal L u\rangle+\frac{w}{\sqrt{\Enl{u}+C}}\langle G(u),u_t\rangle \\
&=\langle u_t,\tilde\mu\rangle
=\langle \Delta \tilde\mu,\tilde\mu\rangle\;=-\int_\Omega |\nabla\tilde\mu|^2\,dx\;\le\;0.
\end{align*}
Hence $E_{\mathrm{SAV}}$ is nonincreasing along solutions of the SAV system.
Moreover, when $w^2=\Enl{u}+C$ holds exactly, $E_{\mathrm{SAV}}(u,w)=E(u)$.

\subsection{Variable-Step BDF2 Discretization}

Let $0 = t_0 < t_1 < \cdots < t_N = T$ be a partition and define
\[
\dt_n=t_{n+1}-t_n,\qquad n=0,\ldots,N-1.
\]
For $n\ge1$, define the step ratio
\[
\rho_n = \frac{\dt_n}{\dt_{n-1}} \in \left[ \rho_{\min}, \rho_{\max} \right ] 
\subset \left(0,1 + \sqrt{2}\right).
\]

A variable-step BDF2 formula approximates $u'(t_{n+1})$ by
\begin{equation}
\label{eq:vBDF2}
\delta_t^{\mathrm{vBDF2}}u^{n+1}
:= \frac{d_n^{(0)} u^{n+1} + d_n^{(1)} u^{n} + d_n^{(2)} u^{n-1}}{\dt_n},
\end{equation}
where
\begin{equation}
\label{eq:BDF-coeffs}
d_n^{(0)} = \frac{1 + 2\rho_n}{1 + \rho_n},
\qquad
d_n^{(1)} = -(1 + \rho_n),
\qquad
d_n^{(2)} = \frac{\rho_n^2}{1 + \rho_n}.
\end{equation}
are the BDF2 coefficients from \cite{hairer1987solving}.

We set, for the step producing $(u^{n+1},w^{n+1})$,
\[
\bar{u}^{\,n+1}:=(1+\rho_n)\,u^{n}-\rho_n\,u^{n-1},
\qquad
q^{n+1}:=\frac{w^{\,n+1}}{\sqrt{\Enl{\bar u^{\,n+1}}+C}},
\]
The SAV scheme is commonly combined with the stabilization technique as is done in \cite{Zhang2020,ZhangC2020FCH} which yields the stabilized SAV (SSAV) scheme. It's shown in \cite{Zhang2020} that the SSAV allows us to take larger time steps and save computational costs. Additionally, with larger time steps the SAV loses accuracy whereas the SSAV scheme doesn't. In \cite{ZhangC2020FCH}, it's shown in the accuracy test for the FCH model that the SAV scheme significantly loses accuracy when $\dt>5\times10^{-4}$ but the SSAV scheme maintains second order accuracy. 

The SSAV--vBDF2 scheme (linearly implicit in $u^{n+1}$) takes the form, for $n\ge1$, as
\begin{align}
\delta_t^{\mathrm{vBDF2}} u^{n+1}
&= \Delta\left(\mathcal{L}u^{n+1} + q^{n+1}\,G(\bar u^{n+1}) + S(u^{n+1}-\bar u^{n+1})\right)
\;=\; \Delta\tilde\mu^{\,n+1},
\label{eq:scheme-u}
\\[2mm]
\delta_t^{\mathrm{vBDF2}} w^{n+1}
&= \frac{1}{2\sqrt{\Enl{\bar u^{n+1}}+C}}\,
\Big\langle G(\bar u^{n+1}),\, \delta_t^{\mathrm{vBDF2}} u^{n+1}\Big\rangle,
\label{eq:scheme-w}
\end{align}
where $S>0$ is a stabilizing parameter and $\tilde\mu^{\,n+1} := \Lop u^{n+1} + q^{n+1} G(\bar u^{n+1})+S(u^{n+1}-\bar u^{n+1})$ is the
discrete chemical potential. 

With $\Delta$ negative semidefinite under
periodic or homogeneous Neumann boundary conditions, this is the
dissipative direction, consistent with
\eqref{eq:gradient_flow}.  Mass conservation comes for free as applying $\Delta$ zeros out changes to the mean.

\section{Energy-Rate Equidistribution}
\label{sec:equidistribution}

The energy dissipation rate $dE/dt$ may vary over many orders of
magnitude during a gradient-flow simulation. During rapid coarsening,
interface reconnection, or defect annihilation the energy can change
quickly, while near equilibrium it may change only very slowly. A
uniform temporal mesh therefore places the same number of time steps in
regions with very different levels of activity. Our goal is to
redistribute the time steps so that more are taken where the energy is
changing rapidly and fewer where it is nearly constant.

This is naturally viewed as a one-dimensional version of
$r$--adaptivity.  Introduce a computational coordinate
$\xi\in[0,1]$ and a strictly increasing map
\[
    t=t(\xi), \qquad t(0)=0,\qquad t(1)=T.
\]
A uniform mesh in $\xi$ then generates a nonuniform mesh in physical
time. Given a positive monitor function $g$, the usual
equidistribution principle is to choose the map so that
\[
    g(t(\xi))\,t'(\xi)=\text{constant}.
\]
Thus regions in which $g$ is large receive smaller physical time
steps, while regions in which $g$ is small receive larger ones
\cite{deboor1987,RusselBook}.

For gradient flows, the energy itself provides a natural quantity from
which to construct the monitor.  For example, the choice
\[
    g(\dot E)=\sqrt{1+\beta^2\dot E^2}
\]
corresponds to equidistributing arclength along the scaled energy curve
\[
    t\longmapsto (t,\beta E(t)).
\]
Other possibilities include
\[
    g(\dot E)=|\dot E|,
    \qquad
    g(\dot E)=\sqrt{|\dot E|},
\]
or combinations of these monitors.  All of them have the same basic
effect: rapid energy change produces a denser temporal mesh, while
nearly flat portions of the energy curve are traversed with larger
steps.

Figure~\ref{fig:Monitors} illustrates this idea for two simple energy
histories.  The first,
\[
    E_1(t)=e^{-t},
\]
has a smoothly decreasing dissipation rate, while
\[
    E_2(t)=\tanh(10-t)
\]
contains a relatively sharp transition between two slowly varying
regimes.  The physical-time plots show the nonuniform distribution of
the time points produced by the monitors.  When the same energy curves
are plotted against the corresponding computational coordinate, the
rapid changes are spread over a larger portion of the computational
interval.  Equidistribution can therefore be viewed as a
reparametrization of time that resolves changes in the energy more
uniformly.

\begin{figure}[!th]
    \centering
    \includegraphics[width=0.97\textwidth]{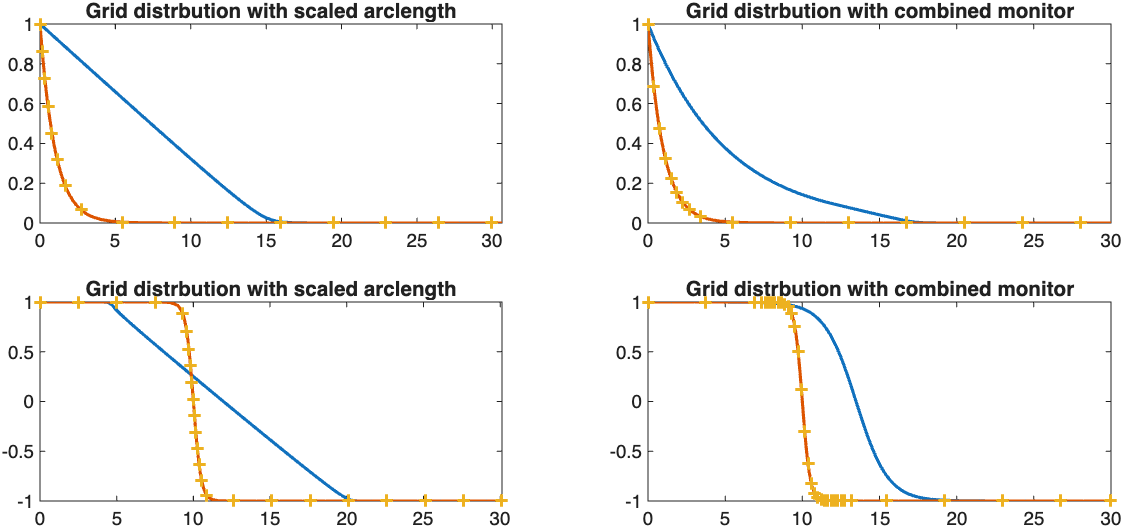}
    \caption{Comparison of a scaled-arclength monitor and a combined
    energy-rate monitor for the energy histories
    $E_1(t)=e^{-t}$ and $E_2(t)=\tanh(10-t)$.  The red curves show the
    energies in physical time and the blue curves show the same energies
    in the corresponding computational coordinate.  The parameters have
    been chosen so that approximately half of the grid points lie in the
    transition region.}
    \label{fig:Monitors}
\end{figure}

The remaining question is which monitor is appropriate for the
adaptive integration of a gradient flow.  We answer this by relating
the energy dissipation rate to the symmetric energy used in our
acceptance criterion.  We first show that the leading discrepancy
between the symmetric and physical energies is proportional to
\[
    \dt\,|\dot E|.
\]
This identifies $g=|\dot E|$ as the natural monitor and leads to an
asymptotic minimax result: equidistributing $|\dot E|$ minimizes the
leading-order maximum symmetric-energy discrepancy for a fixed number
of time steps.  We then consider the behaviour of this monitor near
equilibrium, introduce the combined $|\dot E|$ and
$\sqrt{|\dot E|}$ rule used in the computations, and relate the
scaled-arclength controller of \cite{Qiao2011,Zhang2012} to the same
framework.

Local truncation error (LTE) control, such as the AM3-based controller of \cite{christlieb2024FCH}, adjusts $\Delta t$ to keep a local error estimator below a tolerance, without reference to the energy structure. Our controller adjusts $\Delta t$ to equidistribute the energy dissipation rate, which directly controls the departure of the symmetric energy from the physical energy. For gradient flows, this means the step-size selection is informed by the physics of the problem, not just the numerical error. Both approaches yield second-order accuracy; they differ in which error they control. Section~\ref{sec:examples} compares them on a benchmark problem.

Nothing in this section depends on the SSAV reformulation or on the
BDF2 discretization.  The analysis is first carried out for a smooth
gradient-flow trajectory and its continuous energy.  In
Section~\ref{sec:algorithm}, the continuous dissipation rate is
replaced by a computable discrete estimator and the resulting
controller is coupled to SSAV--vBDF2.

\subsection{Symmetric energy and its relation to dissipation rate}

Define the practical symmetric energy
\begin{equation}
\label{eq:practical-sym-energy}
\Esym^n
:=
\frac12\Big(E(u^n) + E(2u^n - u^{n-1})\Big).
\end{equation}

Note that $2u^n - u^{n-1} = u^n + \delta^n$, where
$\delta^n=u^n-u^{n-1}$. For any sequence $\{u^n\}$ and any twice
continuously Fr\'echet differentiable energy $E$ with bounded second
derivative on the relevant region of state space,
\begin{equation}\label{eq:Esym-defect-general}
\Esym^n-E(u^n)
=
\tfrac12 DE(u^n)[\delta^n]
+
O(\|\delta^n\|^2).
\end{equation}
If $E\in C^3$, the second-order term may be written explicitly as
\[
\Esym^n-E(u^n)
=
\tfrac12 DE(u^n)[\delta^n]
+
\tfrac14D^2E(u^n)[\delta^n,\delta^n]
+
O(\|\delta^n\|^3).
\]

We now look at points sampled from an exact smooth trajectory.
For $u^n=u(t_n)$,
\[
u(t_{n-1})
=
u(t_n)-\dt_{n-1}u_t(t_n)+O(\dt_{n-1}^2),
\]
and hence
\[
\delta^n
=
\dt_{n-1}u_t(t_n)+O(\dt_{n-1}^2).
\]
Substituting this into \eqref{eq:Esym-defect-general} gives
\[
\Esym^n-E(u(t_n))
=
\tfrac12 \dt_{n-1}\frac{dE}{dt}(t_n)
+
O(\dt_{n-1}^2).
\]

\begin{prop}[Energy--dissipation linkage]
\label{prop:Esym-linkage}
Let $u:[0,T]\to X$ be a $C^2$ trajectory and let
$E:X\to\mathbb R$ be twice continuously Fr\'echet differentiable with
bounded second derivative on a neighbourhood of the trajectory. For
$u^n=u(t_n)$, define $\Esym^n$ by
\eqref{eq:practical-sym-energy}. Then
\begin{equation}\label{eq:Esym-linkage}
\Esym^n-E(u(t_n))
=
\tfrac12 \dt_{n-1}\frac{dE}{dt}(t_n)
+
O(\dt_{n-1}^2).
\end{equation}
Consequently, for a gradient flow with $\dot E\le0$,
\[
|\Esym^n-E(u(t_n))|
=
\tfrac12\dt_{n-1}|\dot E(t_n)|
+
O(\dt_{n-1}^2).
\]
\end{prop}

\begin{proof}
Equation~\eqref{eq:Esym-defect-general} gives
\[
\Esym^n-E(u(t_n))
=
\tfrac12DE(u(t_n))[\delta^n]+O(\|\delta^n\|^2).
\]
Using
$\delta^n=\dt_{n-1}u_t(t_n)+O(\dt_{n-1}^2)$ and
\[
DE(u(t_n))[u_t(t_n)]
=
\frac{d}{dt}E(u(t))\bigg|_{t=t_n}
\]
gives \eqref{eq:Esym-linkage}.
\end{proof}

\begin{remark}
The closeness $\Esym^n = E(u^n) + O(h)$ is first-order, not second-order. The leading $O(h)$ term is $\frac12 \dt_{n-1}\,dE/dt$, which does not vanish for solutions undergoing active dissipation.  We exploit this gap for adaptivity. The departure from $E(u^n)$ encodes how fast the energy is changing, and controlling it through the step size is the basis of the equidistribution framework developed below.
\end{remark}

An adaptive controller that enforces $\dt_{n-1} |dE/dt|(t_n) \lesssim \tau_1$ directly bounds this departure by $O(\tau_1)$ while keeping the $O(\dt_{n-1}^2)$ remainder small:
\[
|\Esym^n - E(u^n)| \;\lesssim\; \tfrac12\,\tau_1 + C\,\dt_{n-1}^2.
\]

\subsection{Equidistribution framework}
\label{sec:equidist-theory}

Proposition~\ref{prop:Esym-linkage} shows that the leading symmetric-energy discrepancy is controlled by the product $\dt_{n-1}|dE/dt|$. We now formulate the corresponding minimax problem in terms of a continuous mesh map.

Let $\xi\in[0,1]$ denote a computational coordinate and let
$t=t(\xi)$ be a continuously differentiable, strictly increasing map
with
\[
t(0)=0,\qquad t(1)=T.
\]
A uniform computational mesh $\xi_n=n/N$ induces the physical nodes
$t_n=t(\xi_n)$. If
\[
g(t):=|\dot E(t)|,
\]
then
\[
\dt_n=t_{n+1}-t_n
=
\frac1N t'(\xi_n)+O(N^{-2}),
\]
and Proposition~\ref{prop:Esym-linkage} gives the leading symmetric-energy
discrepancy
\[
|\Esym^n-E(u(t_n))|
=
\frac1{2N}g(t(\xi_n))t'(\xi_n)+O(N^{-2}).
\]

\begin{definition}[Equidistribution \cite{deboor1987}]
A mesh map $t:[0,1]\to[0,T]$ \emph{equidistributes} a positive monitor
function $g$ if
\[
g(t(\xi))\,t'(\xi)=\sigma
\qquad\text{for all }\xi\in[0,1],
\]
where
\[
\sigma=\int_0^T g(s)\,ds.
\]
\end{definition}

\begin{prop}[Asymptotic minimax equidistribution]
\label{prop:equidist-optimal}
Let $E\in C^2([0,T])$ and suppose
\[
g(t)=|\dot E(t)|>0,\qquad 0\le t\le T.
\]
Let $t:[0,1]\to[0,T]$ be a $C^1$, strictly increasing mesh map with
$t(0)=0$ and $t(1)=T$, and let $t_n=t(n/N)$. Then the leading-order
maximum symmetric-energy discrepancy is
\[
\max_n|\Esym^n-E(u(t_n))|
=
\frac1{2N}
\left\|
g(t(\xi))t'(\xi)
\right\|_{L^\infty(0,1)}
+
O(N^{-2}).
\]
Among all such mesh maps, the leading term is minimized by the
equidistribution condition
\begin{equation}\label{eq:continuous-equidist}
g(t(\xi))t'(\xi)
=
\int_0^T g(s)\,ds.
\end{equation}
The minimum leading-order value is
\[
\frac1{2N}\int_0^T|\dot E(t)|\,dt.
\]
For a monotonically decaying energy this becomes
\begin{equation}
\label{eq:minimax-value}
\frac{E(0)-E(T)}{2N}.
\end{equation}
\end{prop}

\begin{proof}
Set
\[
q(\xi):=g(t(\xi))t'(\xi).
\]
Since $t$ is increasing,
\[
\int_0^1q(\xi)\,d\xi
=
\int_0^1g(t(\xi))t'(\xi)\,d\xi
=
\int_0^Tg(t)\,dt.
\]
Because $q\ge0$,
\[
\|q\|_{L^\infty(0,1)}
\ge
\int_0^1q(\xi)\,d\xi
=
\int_0^Tg(t)\,dt.
\]
Equality holds precisely when $q$ is constant, which is
\eqref{eq:continuous-equidist}. Combining this minimax identity with
the leading-order expansion above proves the result. When $\dot E\le0$,
\[
\int_0^T|\dot E(t)|\,dt=E(0)-E(T),
\]
giving \eqref{eq:minimax-value}.
\end{proof}

\begin{remark}
The positivity assumption $|\dot E|>0$ excludes intervals on which the
energy is exactly flat. On such intervals the ideal monitor is
degenerate: at leading order the symmetric-energy discrepancy places no
restriction on the mesh density. This is one reason to regularize the
energy-rate monitor in the practical controller below.
\end{remark}

This result identifies $g(\dot E) = |\dot E|$ as the asymptotically optimal monitor for controlling the leading-order maximum symmetric-energy discrepancy. Additionally, we see why energy-rate--based controllers work as they approximately equidistribute the quantity that determines this discrepancy. Equation~\eqref{eq:minimax-value} is a useful practical guide. Halving the leading-order discrepancy requires doubling the step count. Problems with a large total energy drop are intrinsically more expensive at fixed tolerance, independent of how the drop is distributed in time.

\begin{prop}[Monitor function properties]
\label{prop:monitors}
Let $\dot E_n$ denote the energy dissipation rate associated with step $n$.
\begin{itemize}
\item[\emph{(a)}] The rule $\dt_n \le \tau_1\,/\,|\dot E_n|$ equidistributes the monitor $g = |\dot E|$ and ensures the symmetric-energy discrepancy satisfies
$|\Esym^n - E(u^n)| \le \tau_1/2 + O(\dt_{n-1}^2)$.

\item[\emph{(b)}] The rule $\dt_n \le \tau_2\,/\,\sqrt{|\dot E_n|}$ equidistributes $g = \sqrt{|\dot E|}$. Near equilibrium, where $|\dot E_n|\to 0$ and rule \emph{(a)} places no effective restriction on the step, rule \emph{(b)} caps the growth at $\dt_n = \tau_2/\sqrt{|\dot E_n|}$; the associated energy discrepancy $\tfrac12\tau_2\sqrt{|\dot E_n|}$ still tends to zero, so accuracy is not sacrificed. Away from equilibrium, where $|\dot E_n| \gg (\tau_1/\tau_2)^2$, rule \emph{(b)} is the weaker of the two and is inactive.

\item[\emph{(c)}] Taking $\dt_n = \alpha\min(\tau_1/|\dot E_n|,\;\tau_2/\sqrt{|\dot E_n|})$ with safety factor $\alpha\in(0,1)$ selects rule \emph{(a)} when $|\dot E_n| > (\tau_1/\tau_2)^2$ and rule \emph{(b)} otherwise. This corresponds to the combined monitor
\[
g_{\mathrm{comb}}(\dot E) = \max\!\Big(\,|\dot E|\,/\,\tau_1,\;\;\sqrt{|\dot E|}\,/\,\tau_2\,\Big),
\]
which adapts between optimal discrepancy control during rapid dissipation and controlled step growth near equilibrium.
\end{itemize}
\end{prop}

\begin{proof}
Part (a): substituting $\dt_n = \tau_1/|\dot E_n|$ into the energy--dissipation linkage \eqref{eq:Esym-linkage} gives $|\Esym^n - E(u^n)| = \tau_1/2 + O(h^2)$.

Part (b): if $|\dot E_n| = \varepsilon^2$ with $\varepsilon\to 0$, then $\dt_n = \tau_2/\varepsilon$, which is finite for every $\varepsilon>0$ but unbounded as $\varepsilon \to 0$; the corresponding energy discrepancy is $\frac12 (\tau_2/\varepsilon)\varepsilon^2 = \frac12 \tau_2\varepsilon \to 0$. Rule (a) gives $\tau_1/\varepsilon^2$, which grows faster, so (b) is the binding constraint in this limit.

Part (c): the minimum of two step-size rules corresponds to the maximum of the associated monitors. The crossover occurs when $\tau_1/|\dot E| = \tau_2/\sqrt{|\dot E|}$, i.e.\ $\sqrt{|\dot E|} = \tau_1/\tau_2$, i.e.\ $|\dot E| = (\tau_1/\tau_2)^2$.
\end{proof}

\begin{corollary}[Uniform local discrepancy bound for the combined rule]
\label{cor:uniform-local}
Let $\dt_n = \alpha\min\big(\tau_1/|\dot E_n|,\ \tau_2/\sqrt{|\dot E_n|}\big)$
with $\alpha \in (0,1)$. Then, in \emph{both} regimes,
\[
\dt_n\,|\dot E_n| \;\le\; \alpha\,\tau_1,
\qquad\text{hence}\qquad
|\Esym^n - E(u^n)| \;\le\; \tfrac12\alpha\tau_1 + O(\dt_n^2).
\]
\end{corollary}

\begin{proof}
If $|\dot E_n| > (\tau_1/\tau_2)^2$ then rule (a) is active and
$\dt_n|\dot E_n| = \alpha\tau_1$. Otherwise rule (b) is active and
$\dt_n |\dot E_n| = \alpha\tau_2\sqrt{|\dot E_n|}
\le \alpha\tau_2\,(\tau_1/\tau_2) = \alpha\tau_1$.
\end{proof}

Corollary~\ref{cor:uniform-local}  makes the combined monitor
useful in practice as the second rule caps step growth near equilibrium
without weakening the energy-error guarantee established for the
optimal monitor. It is also the estimate that drives the uniform
physical-energy bound in Proposition~\ref{prop:uniform-physical}.

\begin{prop}[Arclength connection]
\label{prop:arclength}
The regularized step-size rule \cite{Qiao2011}
\begin{equation}
\label{eq:scaledAL-prop}
\dt_n \le \frac{\Delta t_{\max}}{\sqrt{1+\beta^2\,|\dot E_n|^2}},
\qquad
\beta := \frac{\Delta t_{\max}}{\tau_1},
\end{equation}
equidistributes the arclength monitor $g(\dot E) = \sqrt{1+\beta^2\dot E^2}$ along the energy curve $t \mapsto (t, \beta E(t))$. The parameter $\beta$ is thus the ratio of the largest permitted step to the energy-error tolerance of Proposition~\ref{prop:monitors}(a); equivalently, $1/\beta = \tau_1/\Delta t_{\max}$ is the dissipation rate at which the two constraints balance.
\end{prop}

\begin{proof}
Regularize rule (a) as $\dt_n \le \tau_1/\sqrt{\varepsilon^2+|\dot E_n|^2}$ with $\varepsilon>0$. As $|\dot E_n| \to \infty$ this reduces to $\tau_1/|\dot E_n|$, recovering \eqref{eq:adapt1}; as $|\dot E_n| \to 0$ it saturates at $\tau_1/\varepsilon$, which we identify with $\Delta t_{\max}$, so $\varepsilon = \tau_1/\Delta t_{\max} = 1/\beta$. Factoring $\varepsilon$ out of the square root,
\[
\frac{\tau_1}{\sqrt{\varepsilon^2+|\dot E_n|^2}}
= \frac{\tau_1/\varepsilon}{\sqrt{1+|\dot E_n|^2/\varepsilon^2}}
= \frac{\Delta t_{\max}}{\sqrt{1+\beta^2|\dot E_n|^2}},
\]
which is \eqref{eq:scaledAL-prop}. The corresponding equidistribution condition $\dt_n\,\sqrt{1+\beta^2\dot E_n^2} = \text{const}$ distributes time steps uniformly along the arclength of the curve $(t, \beta E(t))$.
\end{proof}

\begin{remark}
Both sides of \eqref{eq:scaledAL-prop} carry units of time, so $\beta\dot E$ is dimensionless. This form also matches the rule used in \cite{Qiao2011,Zhang2012}, where the step is taken as $\Delta t_{\max}$ divided by a square root of the same shape; Proposition~\ref{prop:arclength} identifies their tuning parameter with $\Delta t_{\max}/\tau_1$.
\end{remark}

\subsubsection{Energy--rate--based timestep selection}

Once $u^n$ is available, the newest symmetric-energy difference provides
the discrete rate estimator
\begin{equation}\label{eq:rate-estimator}
D_n
:=
\left|
\frac{\Esym^n-\Esym^{n-1}}{\dt_{n-1}}
\right|.
\end{equation}
This is a computable approximation to the recent energy dissipation rate.

The energy--dissipation linkage motivates the practical adaptive rules
\begin{align}
\dt_n &= \frac{\tau_1}{D_n}, \label{eq:adapt1}\\
\dt_n &= \frac{\tau_2}{\sqrt{D_n}}, \label{eq:adapt2}
\end{align}
where $\tau_1, \tau_2 > 0$ are user--selected tolerances.
In practice one takes
\[
\dt_n = \alpha \min\!\left(
  \frac{\tau_1}{D_n},\;
  \frac{\tau_2}{\sqrt{D_n}}
\right),
\qquad 0 < \alpha < 1,
\]
where $\alpha$ is a safety factor.

Note that $D_n = |\dot{E}(t_n)| + \mathcal{O}(h)$ by Proposition 1. Hence all the same bounds follow with only modified constants.

\section{The Adaptive SSAV--vBDF2 Algorithm}
\label{sec:algorithm}

\subsection{Algorithm}

\begin{algorithm}[H]
\caption{Adaptive SSAV--vBDF2 with energy--rate proposal and symmetric-energy acceptance}
\label{alg:defect-controller}
\begin{algorithmic}[1]
\Require
Solution values $Y^{n-1},Y^{n}$, previous step size $\dt_{n-1}$,
tolerances $\tau_1,\tau_2$, safety factor $\alpha\in(0,1)$,
regularization $\delta>0$, step bounds $\dt_{\min},\dt_{\max}$,
and ratio bounds $\rho_{\min},\rho_{\max}$.

\State \textbf{(Monitor)} Compute the current symmetric energy
\[
\Esym^{n}:=\tfrac12\Big(E(u^{n})+E(2u^{n}-u^{n-1})\Big)
\]
and the most recent discrete energy-rate estimator
\[
D_n:=\left|\frac{\Esym^{n}-\Esym^{n-1}}{\dt_{n-1}}\right|.
\]

\State \textbf{(Propose step)} Set
\[
\dt_{\mathrm{prop}}
:=\alpha\min\!\left(
  \frac{ \tau_1}{\delta+D_n},\;
  \frac{\tau_2}{\sqrt{\delta^2+D_n}}
\right).
\]

\State \textbf{(Enforce bounds)} Set
\[
\dt_n \leftarrow \min(\dt_{\max},\max(\dt_{\min},\dt_{\mathrm{prop}})),
\]
\[
\dt_n \leftarrow \min(\rho_{\max}\dt_{n-1},\max(\rho_{\min}\dt_{n-1},\dt_n)).
\]

\State \textbf{(Attempt step)} Set
\[
\rho_n=\frac{\dt_n}{\dt_{n-1}}
\]
and compute a candidate $Y^{n+1}$ via the SSAV--vBDF2 scheme
\eqref{eq:scheme-u}--\eqref{eq:scheme-w} using step size $\dt_n$.

\State \textbf{(Accept/reject)} Compute
\[
\Esym^{n+1}:=\tfrac12\Big(E(u^{n+1})+E(2u^{n+1}-u^{n})\Big).
\]

\State Set $\dt_n^{\mathrm{floor}} \leftarrow \max\big(\dt_{\min},\;\rho_{\min}\dt_{n-1}\big)$.

\While{$\Esym^{n+1} > \Esym^{n}$ \textbf{and} $\dt_n > \dt_n^{\mathrm{floor}}$}
    \State \textbf{Reject step} and reduce
    \[
    \dt_n \leftarrow \max\big(\dt_n^{\mathrm{floor}},\;\tfrac12\,\dt_n\big).
    \]
    \State Update $\rho_n\leftarrow \dt_n/\dt_{n-1}$.
    \State Recompute the candidate $Y^{n+1}$ with the reduced $\dt_n$.
    \State Recompute $\Esym^{n+1}$.
\EndWhile

\State \textbf{Accept step} and record
\[
\rho_{n}\leftarrow \frac{\dt_n}{\dt_{n-1}}.
\]

\State \Return $Y^{n+1}$, $\dt_n$, $\rho_{n}$.
\end{algorithmic}
\end{algorithm}

The ratio clamp in step 3 is applied against $\dt_{n-1}$, not
$\dt_n$, so that it is a genuine restriction. The rejection loop in steps 7--11 halves down to
$\max(\Delta t_{\min},\rho_{\min}\dt_{n-1})$ rather than to
$\Delta t_{\min}$; without this floor, repeated halving could drive
$\rho_n$ below $\rho_{\min}$ and violate the hypothesis of
Proposition~\ref{prop:global-accuracy}.

\subsection{Energy stability of accepted steps}
\label{sec:stability}

The adaptive controller in Algorithm~\ref{alg:defect-controller} enforces energy monotonicity by construction. We now state this precisely and quantify its implications for the physical energy.

\begin{prop}[Practical monotonicity]
\label{prop:practical-mono}
Algorithm~\ref{alg:defect-controller} accepts step $n$ only if $\Esym^{n+1} \le \Esym^n$, or if $\dt_n$ has reached the floor $\max(\Delta t_{\min},\rho_{\min}\dt_{n-1})$. Therefore, along accepted steps above the floor, the discrete symmetric energy is non-increasing:
\[
\Esym^{n+1} \le \Esym^n.
\]
\end{prop}

\begin{proof}
This is immediate from the rejection loop in Algorithm~\ref{alg:defect-controller}: the while loop rejects and halves $\dt_n$ until either $\Esym^{n+1} \le \Esym^n$ or the floor is reached.
\end{proof}

While Proposition~\ref{prop:practical-mono} concerns the symmetric energy, the quantity of physical interest is $E(u^n)$ itself. The acceptance test is a statement about consecutive values of $\Esym$, and because those form a monotone chain it telescopes; combining it with the local bound of Corollary~\ref{cor:uniform-local} gives a bound on the physical energy that does not accumulate with the number of steps.

\begin{prop}[Uniform physical energy bound]
\label{prop:uniform-physical}
Assume $E \in C^3$ with bounded second and third derivatives on the solution trajectory, and suppose every step $n_0 < n \le N$ is accepted above the floor, so that $\Esym^{n}\le\Esym^{n_0}$. Suppose further that the combined rule of Proposition~\ref{prop:monitors}(c) is in force, so that $h_k|\dot E_k| \le \alpha\tau_1$ for all $k$. Then
\begin{equation}
\label{eq:uniform-physical}
E(u^{n}) \;\le\; E(u^{n_0}) \;+\; \alpha\,\tau_1 \;+\; 2C_0\,\Delta t_{\max}^2
\qquad \text{for all } n_0 < n \le N,
\end{equation}
where $C_0$ depends only on bounds on $E''$ and $E'''$. In particular the
bound is uniform in $n$: the tolerance $\tau_1$ caps the total admissible
rise in the physical energy over the whole integration, not merely the
rise per step.
\end{prop}

\begin{proof}
By Proposition~\ref{prop:Esym-linkage} and Corollary~\ref{cor:uniform-local},
$|\Esym^k - E(u^k)| \le \tfrac12\alpha\tau_1 + C_0\Delta t_{\max}^2$ for every $k$.
Hence
\[
E(u^{n}) \le \Esym^{n} + \tfrac12\alpha\tau_1 + C_0\Delta t_{\max}^2
\le \Esym^{n_0} + \tfrac12\alpha\tau_1 + C_0\Delta t_{\max}^2
\le E(u^{n_0}) + \alpha\tau_1 + 2C_0\Delta t_{\max}^2,
\]
using $\Esym^n \le \Esym^{n_0}$ in the middle step.
\end{proof}

\begin{prop}[Per-step physical quasi-monotonicity]
\label{prop:physical-quasi-mono}
Under the assumptions of Proposition~\ref{prop:uniform-physical}, for any accepted step,
\begin{equation}
\label{eq:physical-quasi-mono}
E(u^{n+1}) \;\le\; E(u^n)
\;+\; \tfrac12\,\alpha\tau_1\Big(1-\rho_n^{-1}\Big)_{\!+}
\;+\; C_1\,\dt_n^2,
\end{equation}
where $(\cdot)_+$ denotes the positive part and $C_1$ depends on bounds on $E''$, $E'''$ and on $\rho_{\max}$. In particular:
\begin{itemize}
\item if $\rho_n \le 1$ (non-increasing steps) the middle term vanishes and $E(u^{n+1}) \le E(u^n) + C_1\dt_n^2$;
\item if $|\rho_n - 1| = O(\dt_n)$, the middle term is $O(\dt_n)$ times $\tau_1$ and the bound is $O(\dt_n^2)$ overall;
\item for a general ratio $\rho_n \in (1,\rho_{\max}]$ the middle term is $O(\tau_1)$ and does \emph{not} vanish with $\dt_n$.
\end{itemize}
\end{prop}

\begin{proof}
By Proposition~\ref{prop:Esym-linkage},
\[
\Esym^n = E(u^n) + \tfrac12 \dt_{n-1}\,\dot E_n + R_n,
\qquad
\Esym^{n+1} = E(u^{n+1}) + \tfrac12 \dt_n\,\dot E_{n+1} + R_{n+1},
\]
with $|R_k| \le C_0 h_k^2$. The acceptance condition $\Esym^{n+1} \le \Esym^n$ gives
\[
E(u^{n+1}) - E(u^n)
\;\le\; \tfrac12\big(\dt_{n-1}\dot E_n - \dt_n \dot E_{n+1}\big) + R_n - R_{n+1}.
\]
Along a smooth trajectory $\dot E_{n+1} = \dot E_n + O(\dt_n)$, and $\dt_{n-1} = \dt_n/\rho_n$, so
\[
\tfrac12\big(\dt_{n-1}\dot E_n - \dt_n \dot E_{n+1}\big)
= \tfrac12 \dt_n \dot E_n\big(\rho_n^{-1}-1\big) + O(\dt_n^2)
= \tfrac12 \dt_n |\dot E_n|\big(1-\rho_n^{-1}\big) + O(\dt_n^2),
\]
where the last equality uses $\dot E_n \le 0$. This is non-positive when $\rho_n \le 1$; otherwise it is bounded by $\tfrac12\dt_n|\dot E_n|\,(1-\rho_n^{-1}) \le \tfrac12\alpha\tau_1(1-\rho_n^{-1})$ by Corollary~\ref{cor:uniform-local}. Absorbing $R_n - R_{n+1}$ and the $O(\dt_n^2)$ remainder into $C_1\dt_n^2$ gives \eqref{eq:physical-quasi-mono}.
\end{proof}

\begin{remark}
Proposition~\ref{prop:physical-quasi-mono} makes explicit that the
per-step deviation from physical energy monotonicity is a due to changes in
step-ratio. It is not a step-\emph{size} effect as it is
proportional to $(1-\rho_n^{-1})_+$ and survives the limit
$\dt_n \to 0$ at fixed $\rho_n > 1$. This is the same phenomenon
that the variable-step $G$-stability defect quantifies. It also explains
why Proposition~\ref{prop:uniform-physical} is the more useful statement
for practice. The individual $O(\tau_1)$ rises cannot accumulate, because
the underlying behaviour of $\Esym$ is monotone.
\end{remark}

\subsection{Convergence}
\label{sec:convergence}

\begin{prop}[Global accuracy]
\label{prop:global-accuracy}
Let $u(t)$ be a sufficiently smooth solution of the gradient flow \eqref{eq:gradient_flow} on $[0,T]$, and let $\{u^n\}$ be the SSAV--vBDF2 approximation produced by Algorithm~\ref{alg:defect-controller}. Suppose the step ratios satisfy $\rho_n \in [\rho_{\min}, \rho_{\max}] \subset (0, 1+\sqrt{2})$ and the step sizes satisfy $\dt_n \in [\Delta t_{\min}, \Delta t_{\max}]$. Then there exists $C>0$, depending on $T$, the smoothness of $u$, and the ratio bounds, such that
\[
\|u^n - u(t_n)\|_H \le C\,\Delta t_{\max}^2 \quad \text{for all } t_n \le T.
\]
\end{prop}

\begin{proof}
The variable-step BDF2 formula \eqref{eq:vBDF2} with coefficients \eqref{eq:BDF-coeffs} has local truncation error $O(\dt_n^3)$ per step for any step ratio $\rho_n > 0$. Convergence of variable-step BDF2 methods under bounded step-ratio constraints is established in \cite{hairer1996solving,IserlesNorsett1991}; the requirement $\rho_n < 1+\sqrt{2}$ ensures zero-stability of the method. Algorithm~\ref{alg:defect-controller} enforces $\rho_n \in [\rho_{\min},\rho_{\max}]$ for every accepted step, so the ratio bounds hold by construction. The result then follows from standard convergence theory for zero-stable, consistent linear multistep methods on variable grids.

The SSAV reformulation introduces an additional consistency error of order $O(\dt_n\,|q^n - 1|^2)$ from the coupling between $u$ and $w$. Along smooth solutions, $q^n = w^n/\sqrt{\Enl{\bar u^{n}}+C} = 1 + O(\dt_n)$, so this contributes only $O(\dt_n^3)$ per step and does not affect the global $O(\Delta t_{\max}^2)$ bound.
\end{proof}

\begin{remark}
Sharper convergence results for variable-step BDF2 that track the dependence on the step-ratio sequence are available. Those results show that the global error constant grows at most polynomially with $T$ under bounded ratios, and that tighter ratio bounds yield smaller constants. For the purposes of this paper, the classical convergence result suffices: the adaptive controller produces a step sequence with bounded ratios, and the scheme is second-order accurate on that sequence.
\end{remark}

\subsection{Time-stepper comparison}
\label{sec:stepper-comparison}

The time steppers $\dt_1$ and $\dt_2$ are given by
\begin{equation*}
    \dt_1 = \frac{ \tau_1}{\delta+D_n} \qquad \dt_2 = \frac{\tau_2}{\sqrt{\delta^2+D_n}}
\end{equation*}
and the combined time stepper uses $\dt_1$ and $\dt_2$, as outlined in Algorithm \ref{alg:defect-controller}. 
We compared the combined time stepper with $\dt_1$, $\dt_2$ and the time steppers from \cite{christlieb2024FCH} and \cite{Qiao2011}. The scaled arc length controller from \cite{Qiao2011} uses the time stepper given by (\ref{eq:scaledAL-prop}). The time stepper AM3 is the adaptive time stepping scheme from \cite{christlieb2024FCH} based on the local truncation error. They used a solution computed with the third-order Adams-Moulton (AM3) scheme to control the error. The comparison is done for a Cahn--Hilliard run where the initial condition is a large collection of cosines on $\Omega\in[-\pi,\pi]$. The initial condition was adjusted so that $u\in[-0.5,0.5]$ and $m=0$. The time stepping methods were compared by observing the error in the energy. To compute the error in the energy we used a run with a constant time step of $\dt=10^{-5}$. The error in Figure \ref{fig:CH_comp_E_err} is computed by 
\begin{equation*}
    \int_0^{50}\frac{|E-E_e|}{1+|E|}\,dt
\end{equation*}
where $E_e$ is the energy from the run with constant time step.
Each time stepper adjusts the adaptivity differently, so the error is plotted against the number of FFT. For this example, $\dt_2$ and the combined stepper do the best followed by $\dt_1$ and AM3. We can also see that $\dt_1$ and the scaled arclength time stepper are very similar which is expected according to Proposition \ref{prop:arclength}. Figure \ref{fig:CH_comp_E_err}, shows that $\dt_1$ doesn't do as well as the other time steppers and the reason is because it doesn't capture the small dip in the energy between $t=31$ and 32 that is shown in Figure \ref{fig:CH_comp_E}. Figure \ref{fig:dt1_dt2_comp} shows $\max_n|\Esym^n - E(u^n)|$ for $\dt_1$, $\dt_2$ and the combined monitor. In this plot, $\dt_1$ is optimal which is consistent with Proposition \ref{prop:equidist-optimal}.

\begin{figure}
    \centering
    \begin{subfigure}[t]{0.8\textwidth}
    \includegraphics[width = \textwidth]{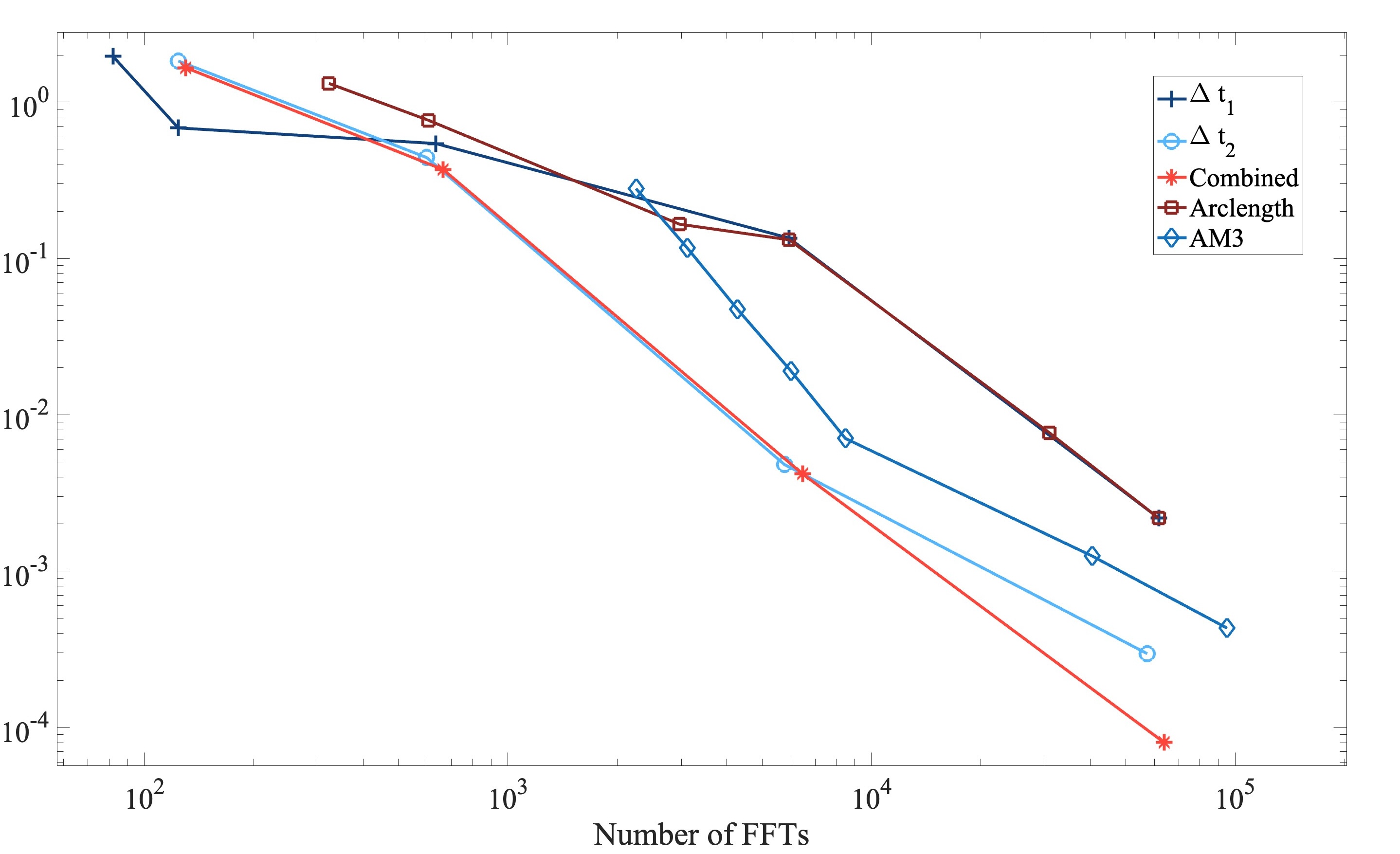}
        \caption{Energy in the error plotted against the number of FFTs. The error in $E(u)$ is computed at $t=50$ against a run with a small constant time step of $1\mathrm{e}{-5}$.}
        \label{fig:CH_comp_E_err}
    \end{subfigure}

    \begin{subfigure}[t]{0.49\textwidth}
    \includegraphics[width = \textwidth]{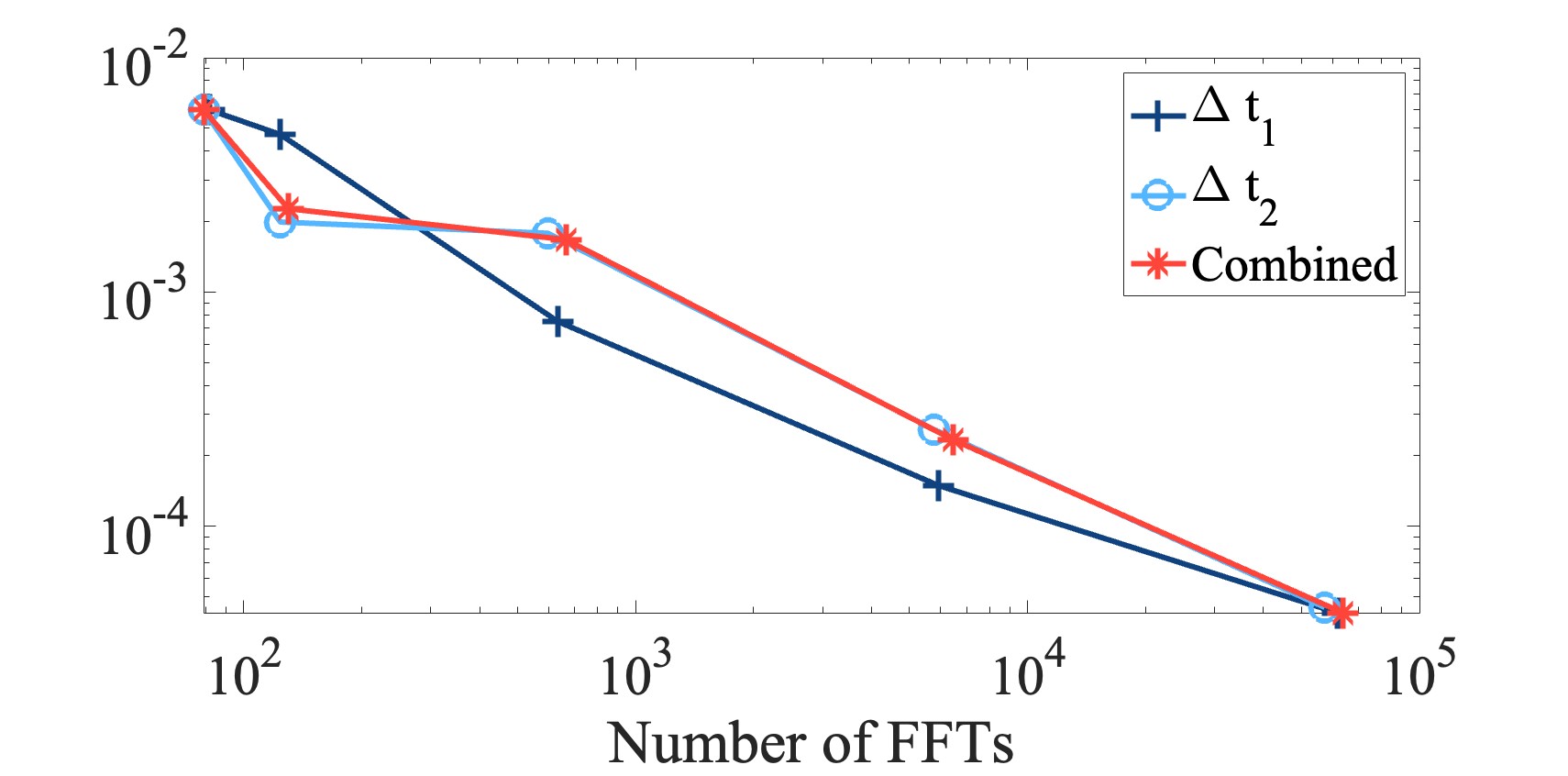}
    \caption{Plot of $\max_n|\Esym^n - E(u^n)|$ for the monitors $\dt_1$, $\dt_2$ and both combined.}
    \label{fig:dt1_dt2_comp}
    \end{subfigure}
    \begin{subfigure}[t]{0.49\textwidth}
        \includegraphics[width=\textwidth]{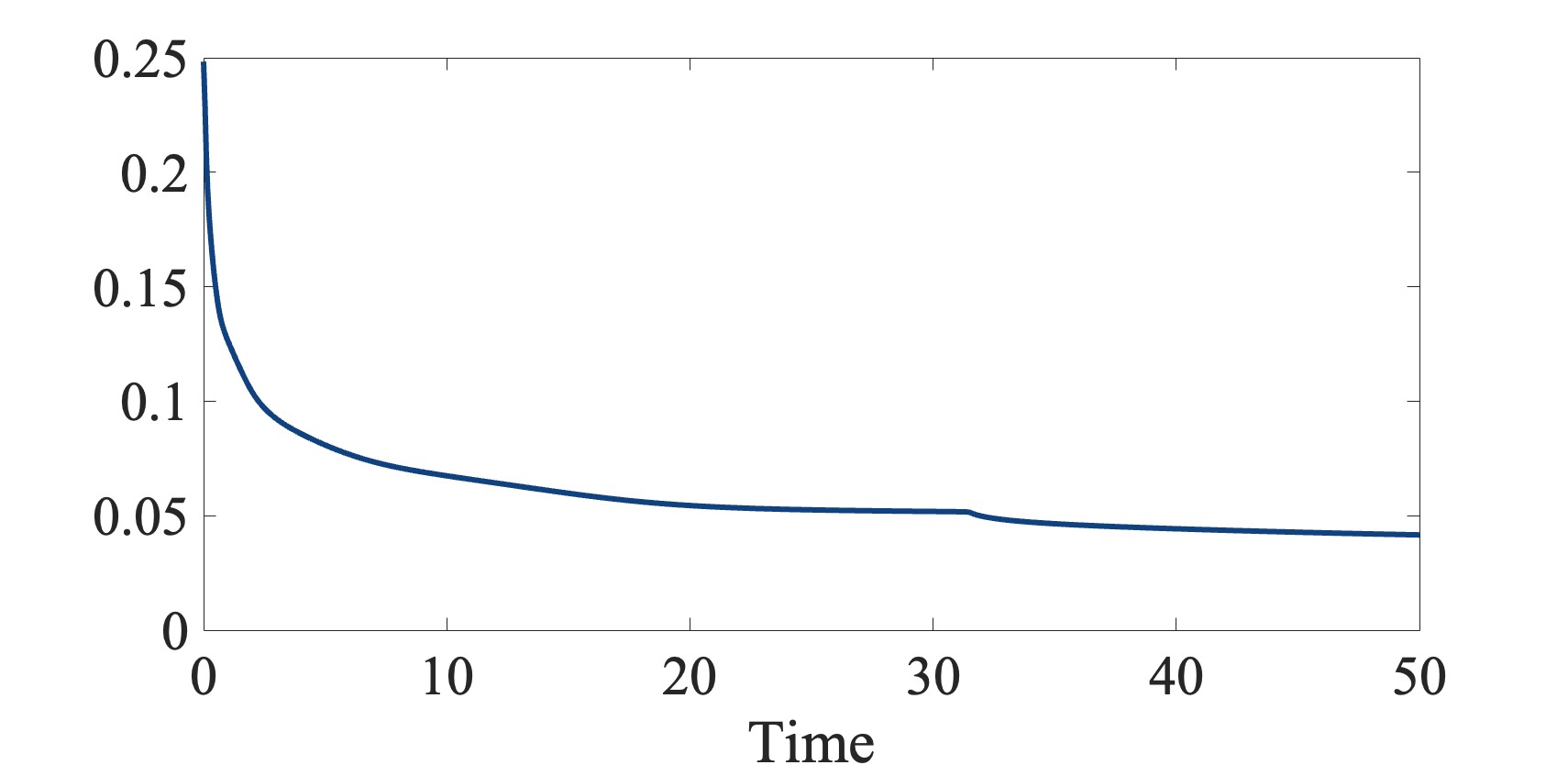}
        \caption{Energy}
        \label{fig:CH_comp_E}
    \end{subfigure}
        \caption{A comparison of the time steppers for a Cahn--Hilliard run.}
\label{fig:CH_comp}
\end{figure}

\section{Numerical Experiments}
\label{sec:examples}

\subsection{Implementation details}

Our full time-step selection algorithm is given in Algorithm \ref{alg:defect-controller}. We define the minimum time step $\dt_{\min}$ and the time stepper enforces a maximum time step of $\dt_{\max}=\tau/\delta$, so the physical step sizes satisfy $\dt_n\in[\dt_{\min},\dt_{\max}]$.
Additionally, a minimum and maximum time step ratio, $\rho_{\min},\rho_{\max}$ are defined to avoid drastically changing the time step. For simplicity, we set the tolerances equal to each other, $\tau=\tau_1=\tau_2$; note that with this choice the crossover of Proposition~\ref{prop:monitors}(c) occurs at $|\dot E|=1$. Given initial values $t_0$ and $Y^0$, we set $\dt_0=\dt_{\min}$ and compute the solution values at $t=t_1$ using BDF1. In the examples below, the time stepper usually doesn't reach the maximum time step and the minimum only during the initial transient phase when starting with random initial data. The parameter values, domain, minimum time step and tolerance used for each example is given in Table \ref{table:model_setup}. For the spatial discretization we use Fourier pseudo-spectral methods with $128^2$ or $256^2$ modes for the 2D runs and $64^3$ for 3D. Additionally, we fix the SAV parameter $C=1$ for all the models except FCH, where for FCH $C=256^2$ which was taken from \cite{ZhangC2020FCH}. For all the models except FCH the stabilization parameter comes from \cite{Zhang2020}. For FCH, the authors in \cite{ZhangC2020FCH} used $S=3\varepsilon^2$ but this choice wasn't adequent and through experimentation we found $S=10$ was a reasonable choice.
The models we consider are subject to the following mass constraint
\begin{equation}
    m=\frac{1}{|\Omega|}\int_\Omega u  \,dx. \nonumber
\end{equation}
In three of the examples below, we start from random initial data, so the initial mass can differ slightly from $m$. We adjust the initial condition to ensure it has the correct mean. For the examples below this mass constraint was fulfilled.

\begin{table}[ht]
\centering
\begin{tabular}{c c c c c c}
 \hline
  & Parameters & $\Omega$ & Resolution &$\dt_{\min}$ & $\tau$\\
 \hline
 && &&\\
 CH & $\begin{array}{c}
      \eps=0.06 \\
      S=2
 \end{array}$  & $[-\pi,\pi]^2$ & $128^2$ & $10^{-5}$ & $10^{-5}$\\
 \hline
 && &&\\
 OK & $\begin{array}{c}
 \gamma=3.5\\
 S=2
 \end{array}$ & $[0,4\pi]^2$ & $128^2$ &$10^{-5}$ & $10^{-5}$\\
 \hline
 && &&\\
 PFC & $\begin{array}{c}
      \beta=0.2 \\
    m=0.25\\
    S=2
 \end{array}$ &  $[0,8\pi]^3$ & $64^3$ &$10^{-2}$ & $10^{-3}$\\
 \hline
 && &&\\
 FCH & $\begin{array}{c}
      \eps=0.1 \\
    \eta=\eps^2 \\
    S=10
 \end{array}$ & $[0,4\pi]^2$ & $256^2$ &$10^{-5}$ & $10^{-5}$\\
\end{tabular}
\caption{Set up for each example}
\label{table:model_setup}
\end{table}

\subsection{Convergence study with exact solution}

To show that the adaptive time stepper is second order, we conduct a convergence study for the 2D Cahn--Hilliard model.
We will use the following initial condition
\begin{equation*}
    u_0=0.2(\text{rand}(N_x,N_y)-1/2)
\end{equation*}
on the domain $\Omega=[-\pi,\pi]^2$ with $64^2$ Fourier modes where $\text{rand}(N_x, N_y)$ is a random number drawn from the uniform distribution in $(0,1)$. 
We vary $\tau$ and set $\dt_{\min}=\tau/100$.
The error is computed against a run with $\tau = 1.2207\mathrm{e}-06$. The error in $u$ is computed by
\begin{equation*}
    e_u = \sqrt{\Delta_x \Delta_y\sum_{i=1}^{N_x}\sum_{j=1}^{N_y}|u-u_E|^2}
\end{equation*}
We conducted 5 runs with different seeds and took the average of the errors across those runs. The error in $u$ and $\max|\Esym-E|$ are presented in Table \ref{table:CH_exact_err} and a plot of the errors is shown in Figure \ref{fig:conv_study}. From the plot and table we can see second order convergence for $u$ and first order convergence for $\max|\Esym-E|$.

\begin{figure}
    \centering
    \includegraphics[width=0.7\linewidth]{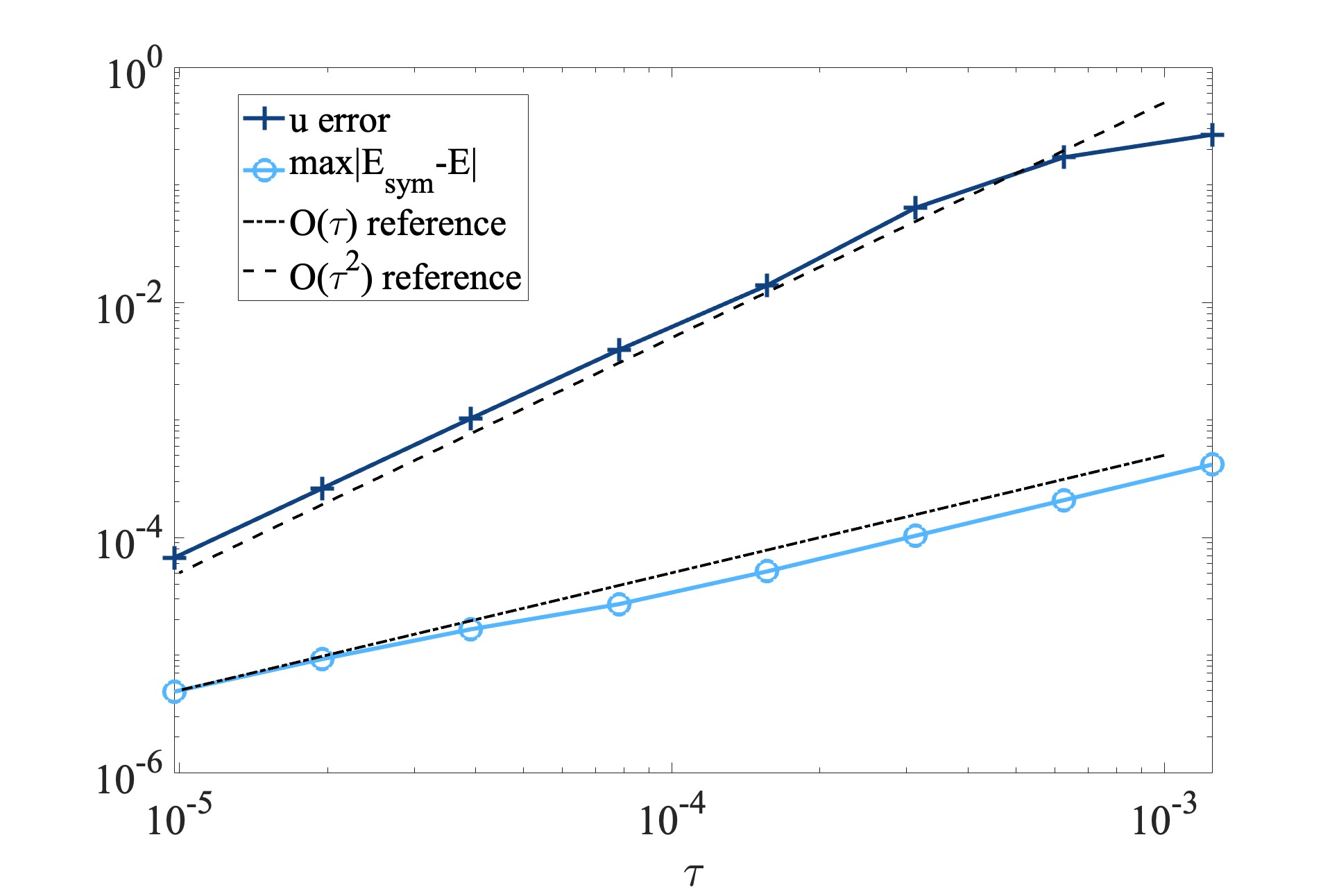}
    \caption{Plot of the error in $u$ and $\max|\Esym-E|$ at $t=1$.}
    \label{fig:conv_study}
\end{figure}

\begin{table}[ht]
\centering
\begin{tabular}{|c | c c |c c |c|}
 \hline
 $\tau$ & Error in $u$ & Rate & $\max|\Esym - E(u)|$ & Rate & Rejected steps\\
 \hline
 $1.250\mathrm{e}{-03}$ & $2.6694\mathrm{e}{-01}$ & $-$ & $4.1668\mathrm{e}{-04}$ & - & 4.4 \\
 $6.2500\mathrm{e}{-04}$ & $1.7128\mathrm{e}{-01}$ & $0.6402$ & $2.0752\mathrm{e}{-04}$ & 1.0057 & 6.0\\
 $3.125\mathrm{e}{-04}$ & $6.3639\mathrm{e}{-02}$ & $1.4283$ & $1.0332\mathrm{e}{-04}$ & 1.0062 & 4.0\\
 $1.5625\mathrm{e}{-04}$ & $1.3963\mathrm{e}{-02}$ & $2.1883$ & $5.1431\mathrm{e}{-05}$ & 1.0062 & 3.6\\
 $7.8125\mathrm{e}{-05}$ & $3.9387\mathrm{e}{-03}$ & $1.8259$ & $2.7076\mathrm{e}{-05}$ & 0.9258 & 3.8\\
 $3.9063\mathrm{e}{-05}$ & $1.0264\mathrm{e}{-03}$ & $1.9401$ & $1.6537\mathrm{e}{-05}$ & 0.7113 & 2.0\\
 $1.9531\mathrm{e}{-05}$ & $2.6229\mathrm{e}{-04}$& $1.9684$ & $9.2386\mathrm{e}{-06}$ & 0.8399 & 3.6\\
 $9.7656\mathrm{e}{-06}$ & $6.7136\mathrm{e}{-05}$ & $1.9660$ & $4.8777\mathrm{e}{-06}$ & 0.9215 & 3.6\\
 \hline
\end{tabular}
\caption{Errors at $t=1$ along with the rate of convergence.}
\label{table:CH_exact_err}
\end{table}

\subsection{Cahn--Hilliard spot collapse}
\label{sec:CH}

The Cahn--Hilliard model has the following free energy and PDE
\begin{equation}
    E(u) = \frac{1}{|\Omega|}\int_\Omega\left(\frac{\varepsilon^2}{2}|\nabla u|^2+\frac{1}{4}(1-u^2)^2\right)\nonumber \text{.}
\end{equation}
\begin{equation}
    u_t = \Delta(-\varepsilon^2\Delta u + u^3-u)\nonumber
\end{equation}
For the Cahn-Hilliard model, we will look at the spot collapse example from \cite{Zhang2020}. The initial condition is two circles with different radii given by the following equation
\begin{equation}
\label{eq:CH-IC}
    u_0=-\left[\tanh\left(\frac{\sqrt{(x+0.8)^2+y^2}-1.4}{1.5\varepsilon}\right) + \tanh\left(\frac{\sqrt{(x-1.7)^2+y^2}-0.5}{1.5\varepsilon}\right)\right]+1
\end{equation}

The two tanh factors each transition over a width proportional to $\varepsilon$. The shift and rescaling in \eqref{eq:CH-IC} are chosen so that $u \approx 1$ inside either circle and $u \approx -1$ everywhere else. The energy, time step and profiles of $u$ are plotted in Figure \ref{fig:CH_spot_collapse}. In Figure \ref{fig: CH_spot_collapse_dt}, we see that the time step follows the behaviour of the energy. At around $t=19$, there is a rapid decrease in the energy so the time step decreases. As the energy flattens out after $t=20$, there is an increase in the time step.

\begin{figure}
\centering
    \begin{subfigure}{0.49\textwidth}
        \includegraphics[width=\textwidth]{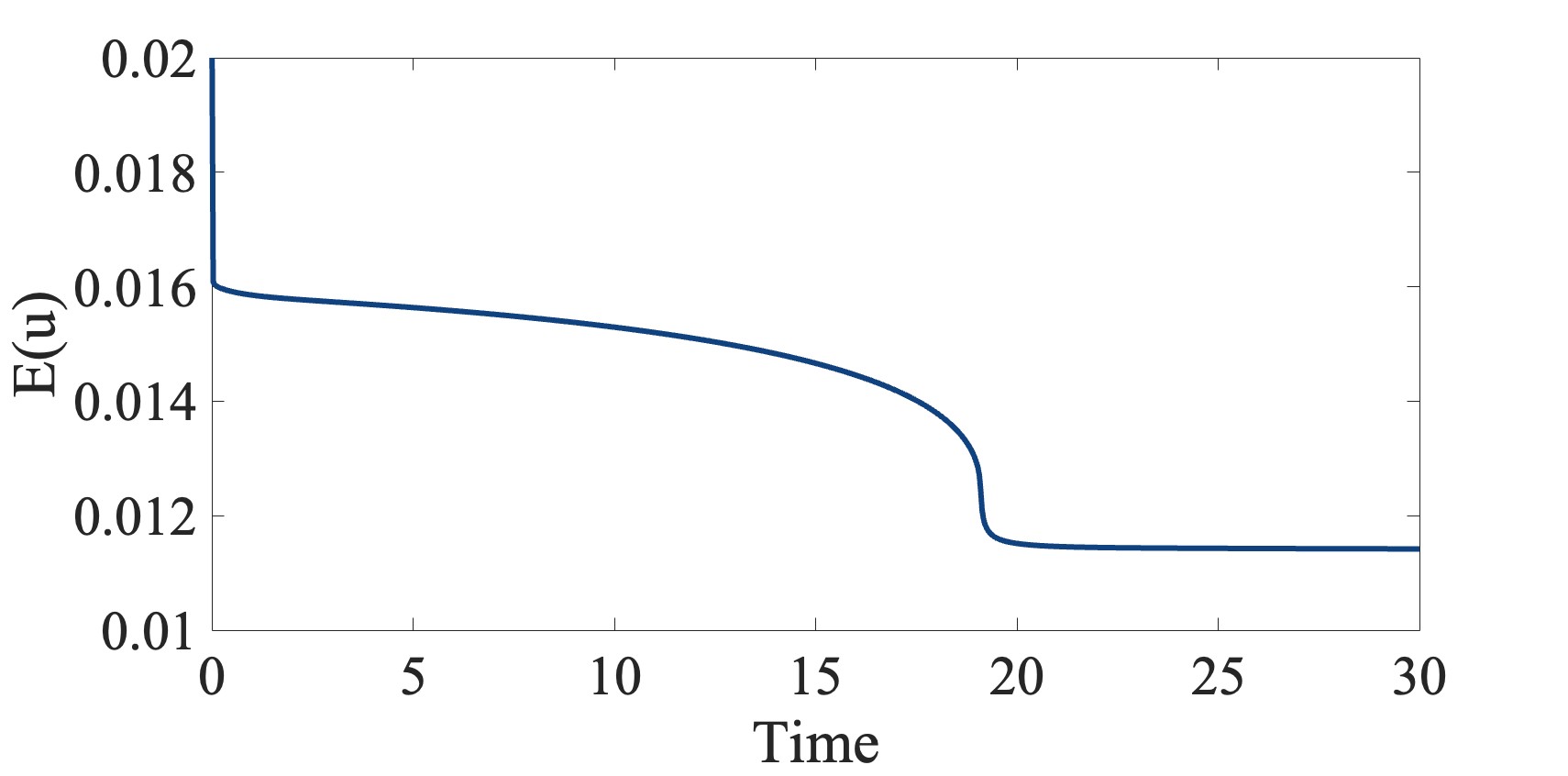}
        \caption{Energy}
        \label{fig: CH_spot_collapse_E}
    \end{subfigure}
    \begin{subfigure}{0.49\textwidth}
        \includegraphics[width=\textwidth]{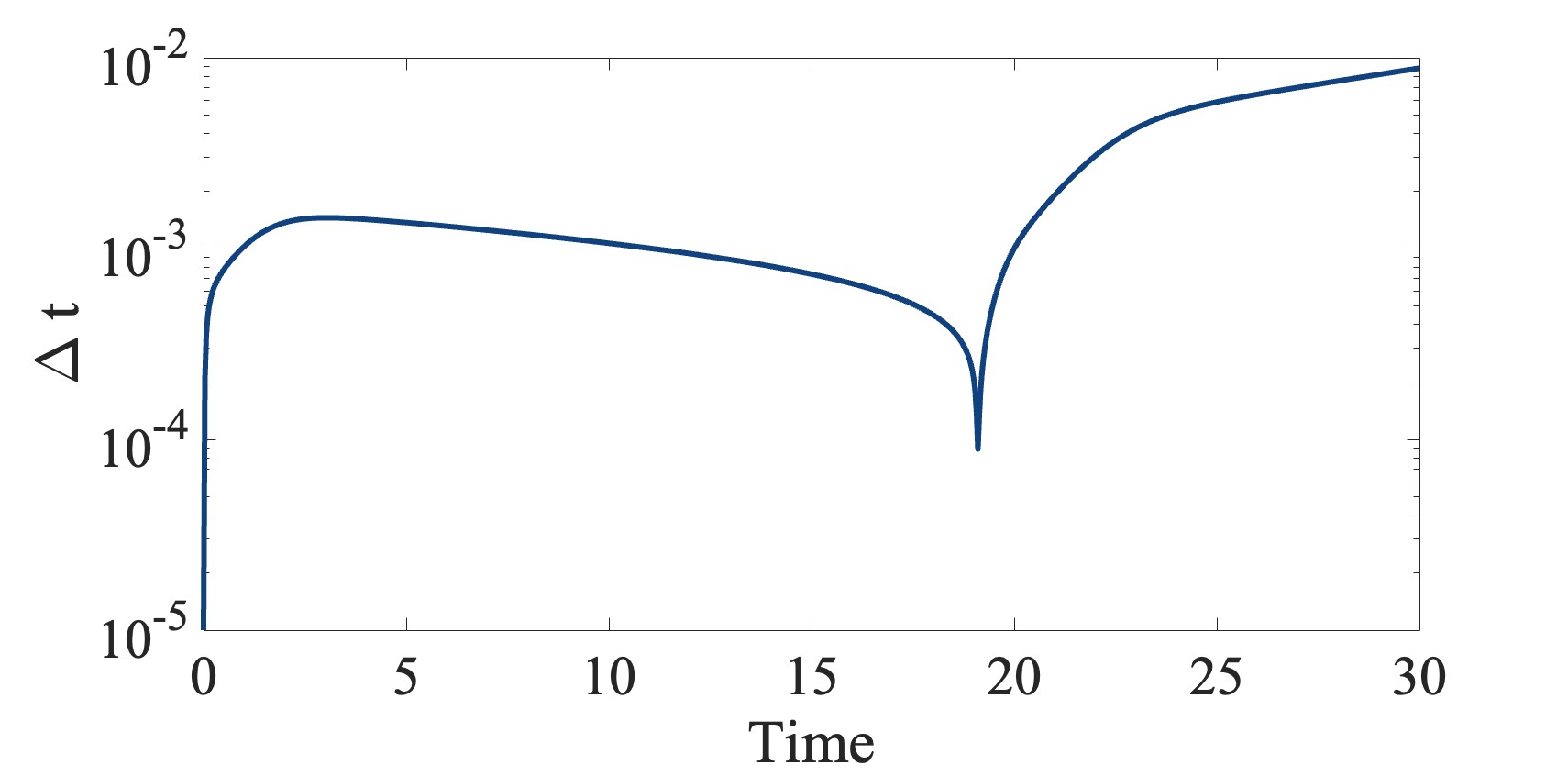}
        \caption{Time step}
        \label{fig: CH_spot_collapse_dt}
    \end{subfigure}

    \begin{subfigure}{0.99\textwidth}
        \includegraphics[width=\textwidth]{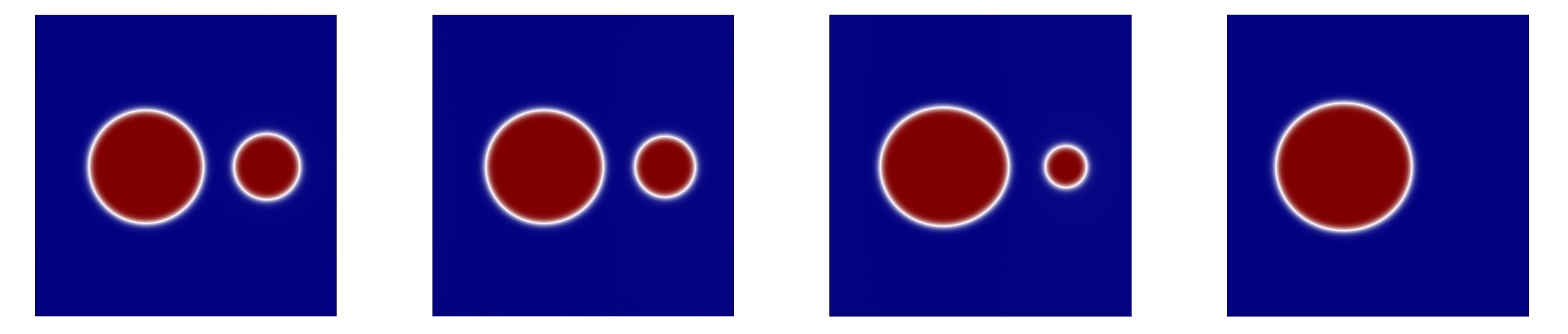}
        \caption{Profiles of $u$ at times $t=1$, 5, 15 and 30.}
        \label{fig: CH_spot_collapse_u}
    \end{subfigure}
    \caption{Plots of energy, time step, and profiles of $u$ for the 2D Cahn-Hilliard that models the spot collapse.}
    \label{fig:CH_spot_collapse}
\end{figure}

\subsection{Ohta--Kawasaki mixed states}

The next model we will look at is the Ohta--Kawasaki model, which is the nonlocal extension of the Cahn--Hilliard model. The free energy and PDE are given by
\begin{equation}
    E(u) = \frac{1}{|\Omega|}\int_\Omega \left(\frac{1}{2\gamma^2}|\nabla u|^2 + \frac{1}{4}(1-u^2)^2+\frac{1}{2}|\nabla v|^2 \right)\,\mathrm{d} x\nonumber
\end{equation}
\begin{equation}
    u_t = \Delta\left(-\frac{1}{\gamma^2}\Delta u + u^3-u +  v\right)\nonumber
\end{equation}
where $v$ is given by the boundary value problem $-\Delta v = u-m$ with periodic boundary conditions on $\partial\Omega$.
For the Ohta-Kawasaki run, we use random initial condition
\begin{equation*}
    u_0=m+0.002(\text{rand}(N_x,N_y)-1/2)
\end{equation*}
where $\text{rand}(N_x, N_y)$ is a random number drawn from the uniform distribution in $(0,1)$.
Results are shown in Figure \ref{fig:OK_mixed_state} for various $m$ values. In Figure \ref{fig:OK_mixed_E}, for all three runs, the energy starts out flat so there is an increase in the time step. The time step decreases when the energy drops and increases again when the energy flattens out. For these runs, there were some early rejected steps due to using random initial data. The number of rejected steps for each run is given in Table \ref{table:rejected_steps}. The rejected steps occurred at the beginning of the run and for the 3 runs we have a negligible rejection rate.

\begin{table}[ht]
\centering
  \begin{tabular}{|c |c |c |c |c|} 
  \hline
   \multicolumn{2}{|c|}{Model} & Attempted steps & Rejected steps  & Rejection rate \\
  \hline
  \multirow{3}{*}{OK} & $m=0.05$ & 92901& 9& 0.0097\% \\
  & $m=0.15$ & 65981& 11& 0.0167\%\\
  & $m=0.25$ &55375 & 3& 0.0054\%\\
  \hline
  \multicolumn{2}{|c|}{PFC} & 1049 & 250 & 23.8322\%\\
  \hline
  \end{tabular}
  \caption{Number of rejected steps and the rejection rate for each model. The Cahn--Hilliard and functionalized Cahn--Hilliard run didn't have any rejected steps.}
  \label{table:rejected_steps}
  \end{table}

\begin{figure}
\centering
    \begin{subfigure}{0.99\textwidth}
        \includegraphics[width=\textwidth]{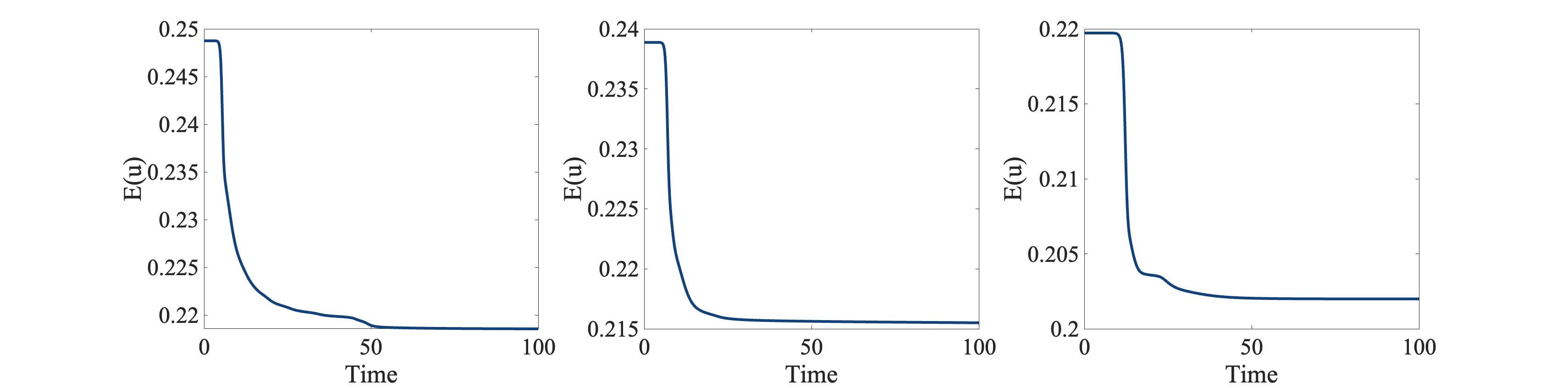}
        \caption{Energy}
        \label{fig:OK_mixed_E}
    \end{subfigure}

    \begin{subfigure}{0.99\textwidth}
        \includegraphics[width=\textwidth]{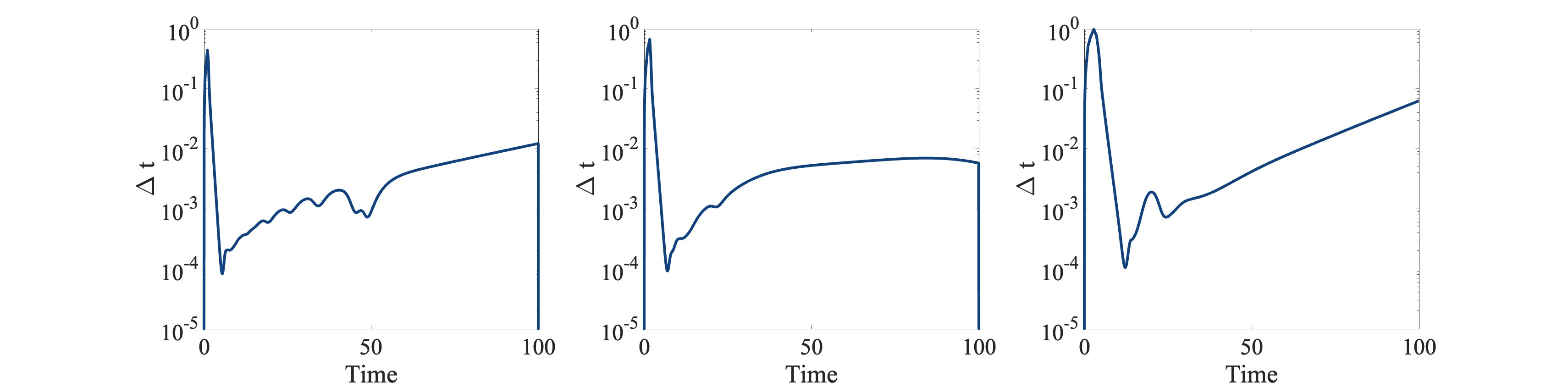}
        \caption{Time step}
        \label{fig:OK_mixed_dt}
    \end{subfigure}

     \begin{subfigure}{0.99\textwidth}
        \includegraphics[width=\textwidth]{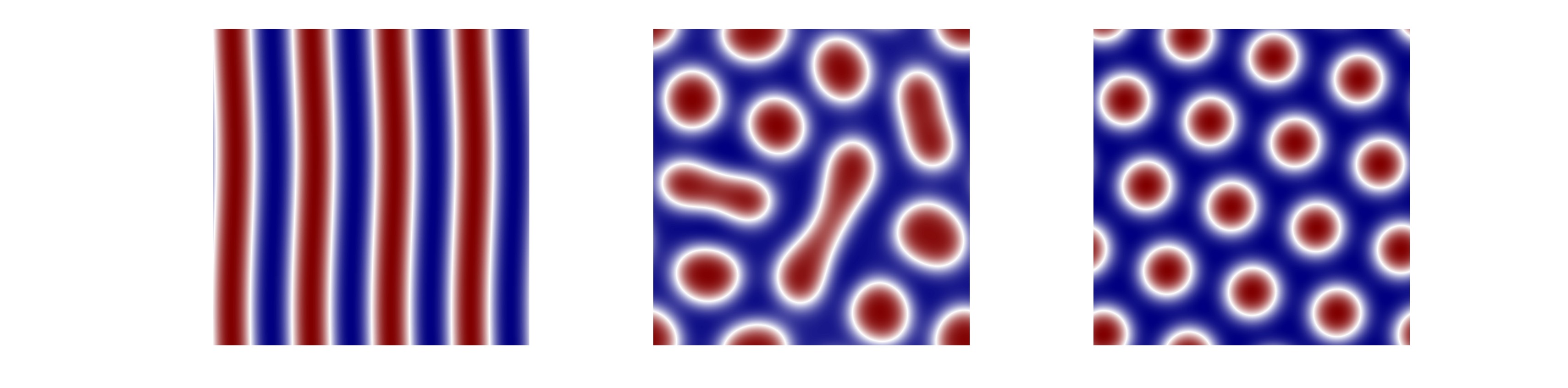}
        \caption{Profiles of $u$}
    \end{subfigure}
    \caption{Plots of energy, time step, and final profiles of $u$ for the 2D Ohta-Kawasaki model for $m=0.05$, $0.15$ and $0.25$.}
    \label{fig:OK_mixed_state}
\end{figure}

\subsection{Phase--Field Crystal in 3D}
\label{sec:PFC}

The PFC free energy for this model is given by
\begin{equation}
    E(u)= \frac{1}{|\Omega|}\int_\Omega \left(\frac{1}{2}|(\Delta+1)u|^2+\frac{1}{4}(u^2-\beta)^2\right) \, dx \nonumber
\end{equation}
where $\beta$ is a parameter. For the PFC run we use the following random initial condition
\begin{equation*}
    u_0=m+0.2(\text{rand}(N_x,N_y,N_z)-1/2)
\end{equation*}
where $\text{rand}(N_x, N_y,N_z)$ is a random number drawn from the uniform distribution in $(0,1)$.
The results are shown in Figure \ref{fig:PFC_rand_IC_adap_dt_3D_00}. In Figure \ref{fig:3D_PFC_E}, for $t\in[500,2000]$ the energy appears flat on the scale of the plot, but the time step continues to fluctuate. These fluctuations are driven by corresponding but inversely related variations in $|dE/dt|$, which ranges between $10^{-9}$ and $10^{-5}$ in this interval. When $|dE/dt|$ increases, the controller reduces the step size, and vice versa. Similar to the Ohta-Kawasaki example, due to using random initial data along with a coarser grid we have some rejected steps. This run had 250 rejected steps out of 1049 attempted steps, which yields a rejection rate of 24.2682\%.
Figure \ref{fig:3D_PFC_reject_time} depicts the number of rejected steps that occurs per time step, and this shows that there is at most 1 rejected step per time step. Figure \ref{fig:3D_PFC_reject_total} shows that the vast majority of rejected steps, 246 out of a total of 250, are concentrated at the start of the run, within the time interval $t\in[0,2.5]$, 

\begin{figure}
    \centering
    \begin{subfigure}[t]{0.49\textwidth}
    \includegraphics[width = \textwidth]{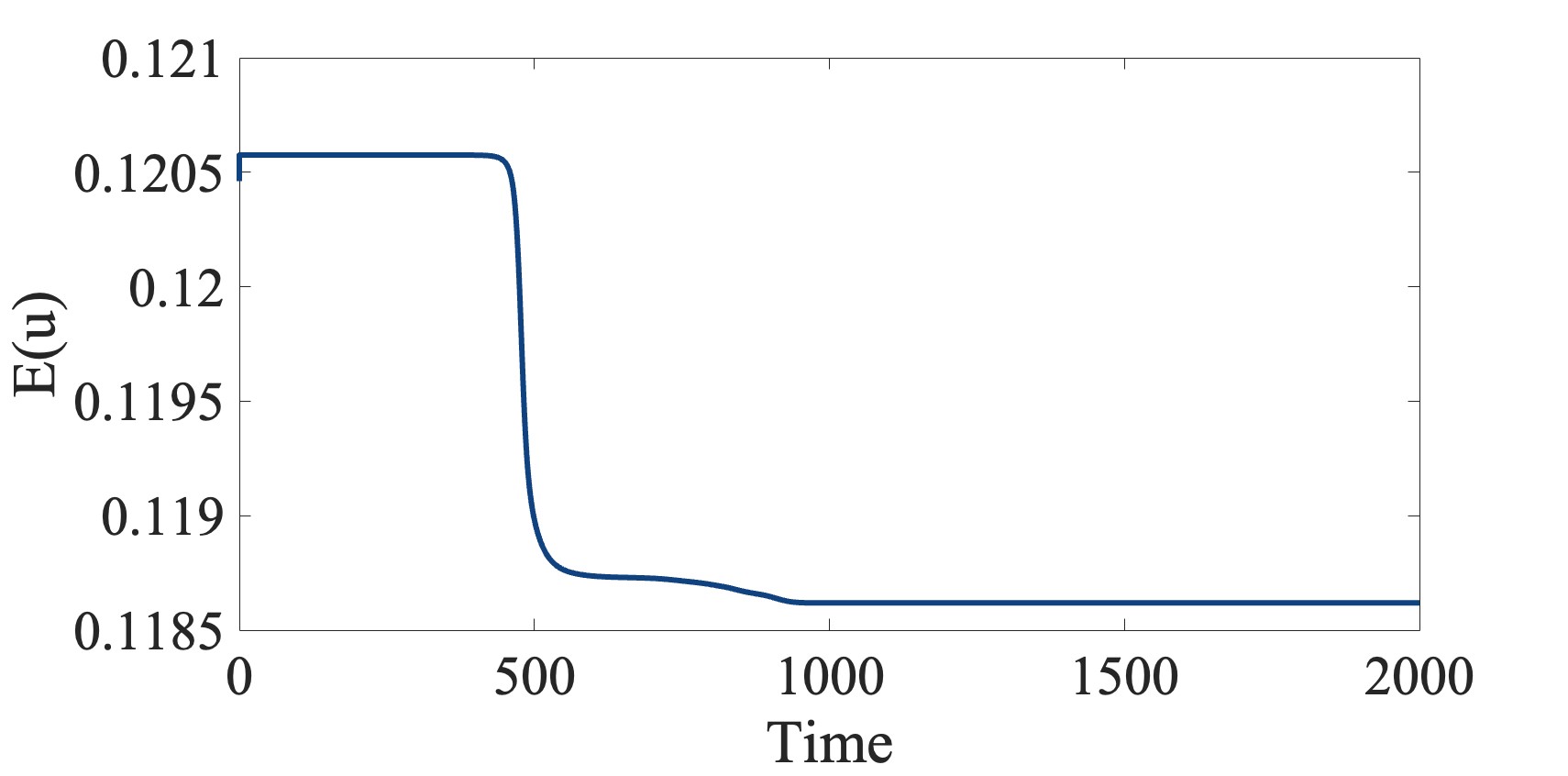}
    \caption{Energy}
    \label{fig:3D_PFC_E}
    \end{subfigure}
    \begin{subfigure}[t]{0.49\textwidth}
    \includegraphics[width = \textwidth]{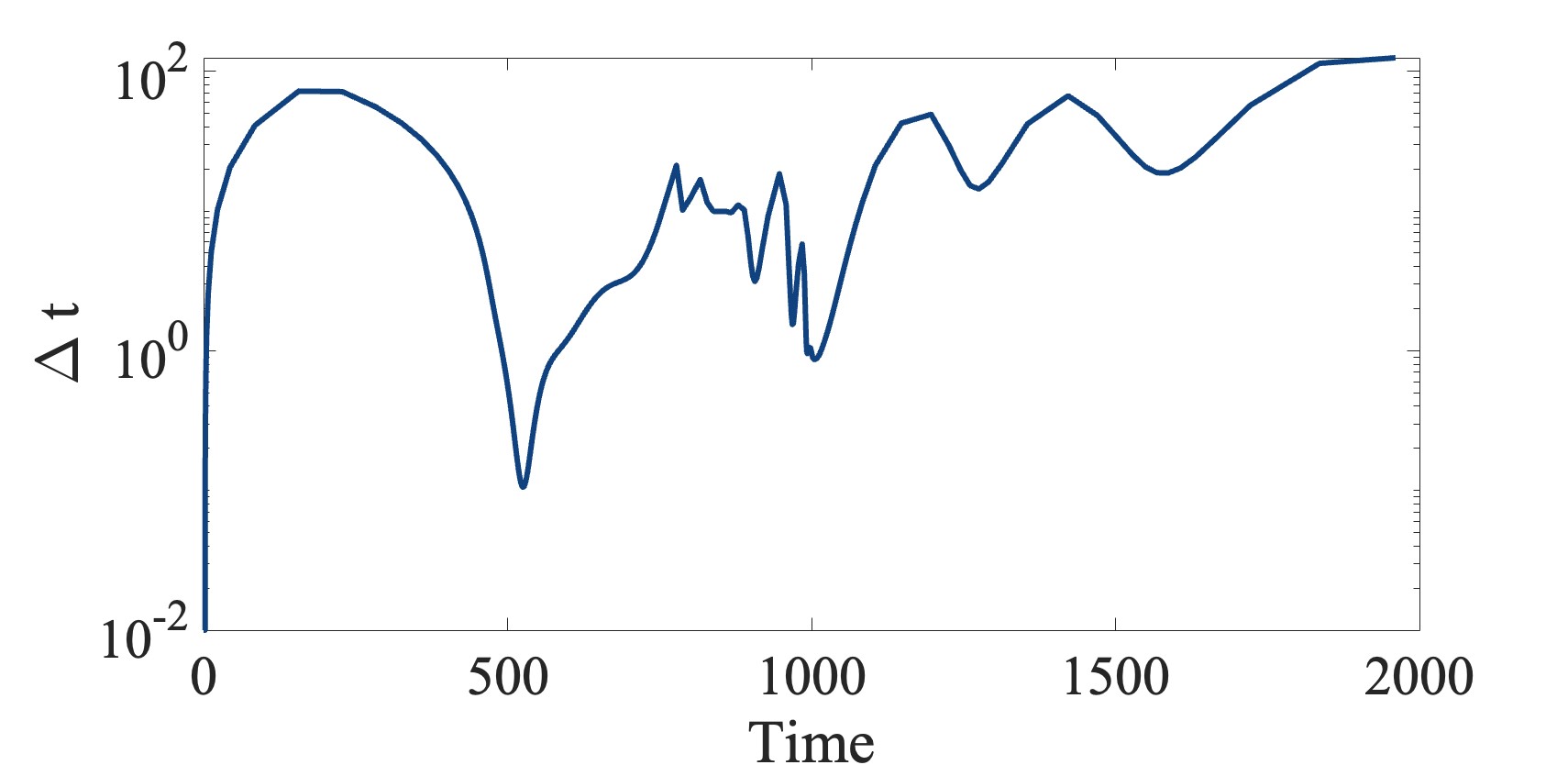}
    \caption{Time step}
    \label{fig:3D_PFC_dt}
    \end{subfigure}

    \begin{subfigure}[t]{0.5\textwidth}
    \includegraphics[width = \textwidth]{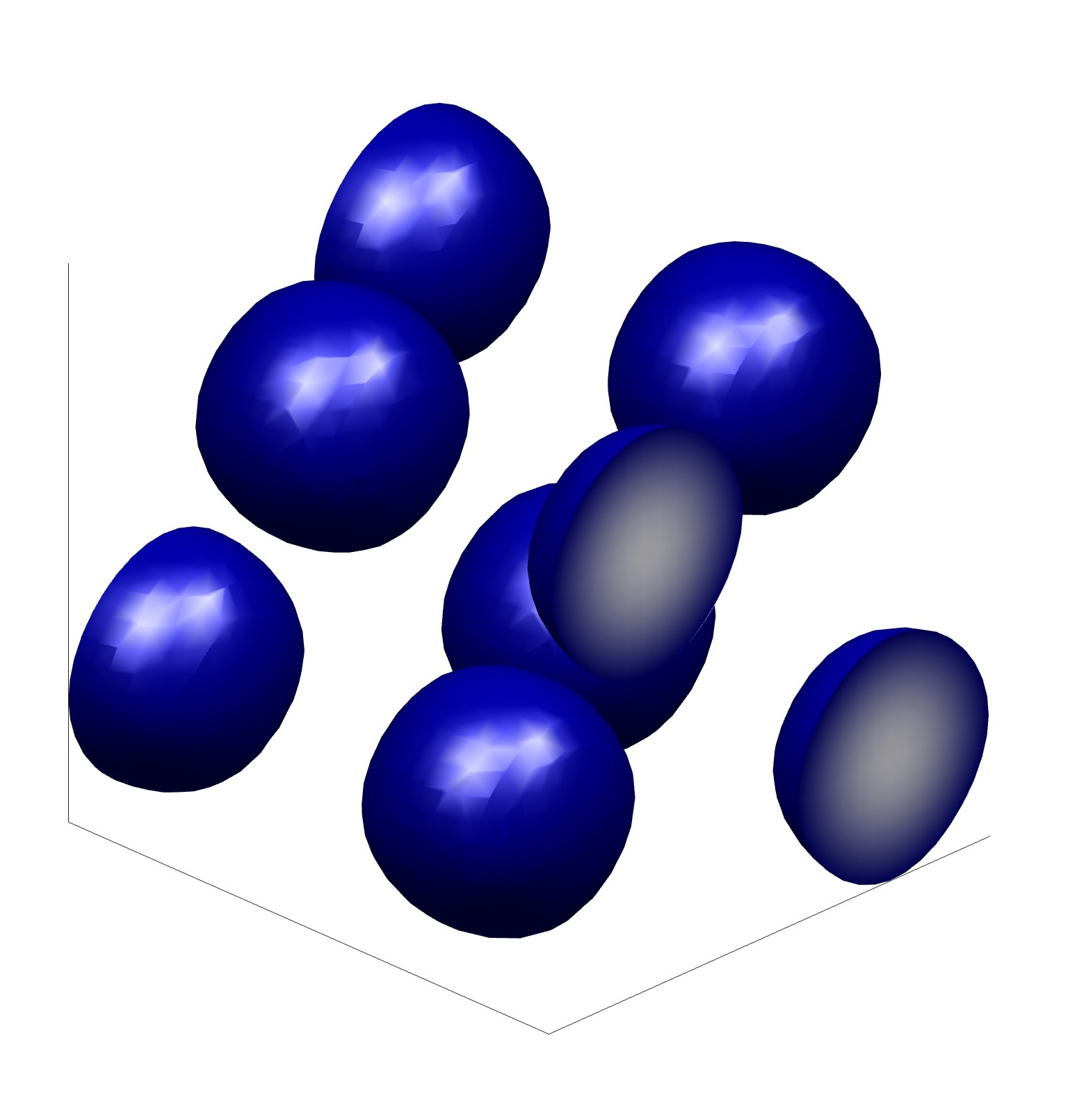}
    \caption{Profile of $u$ at $t=2000$. For this set of parameters, we see a spherical packing.}
    \label{fig:PFC_3D}
    \end{subfigure}
    \caption{Plots of energy, time step, and final profile of $u$ for the 3D PFC model with $m=0.25$.}
    \label{fig:PFC_rand_IC_adap_dt_3D_00}
\end{figure}

\begin{figure}
    \centering
    \begin{subfigure}[t]{0.49\textwidth}
    \includegraphics[width = \textwidth]{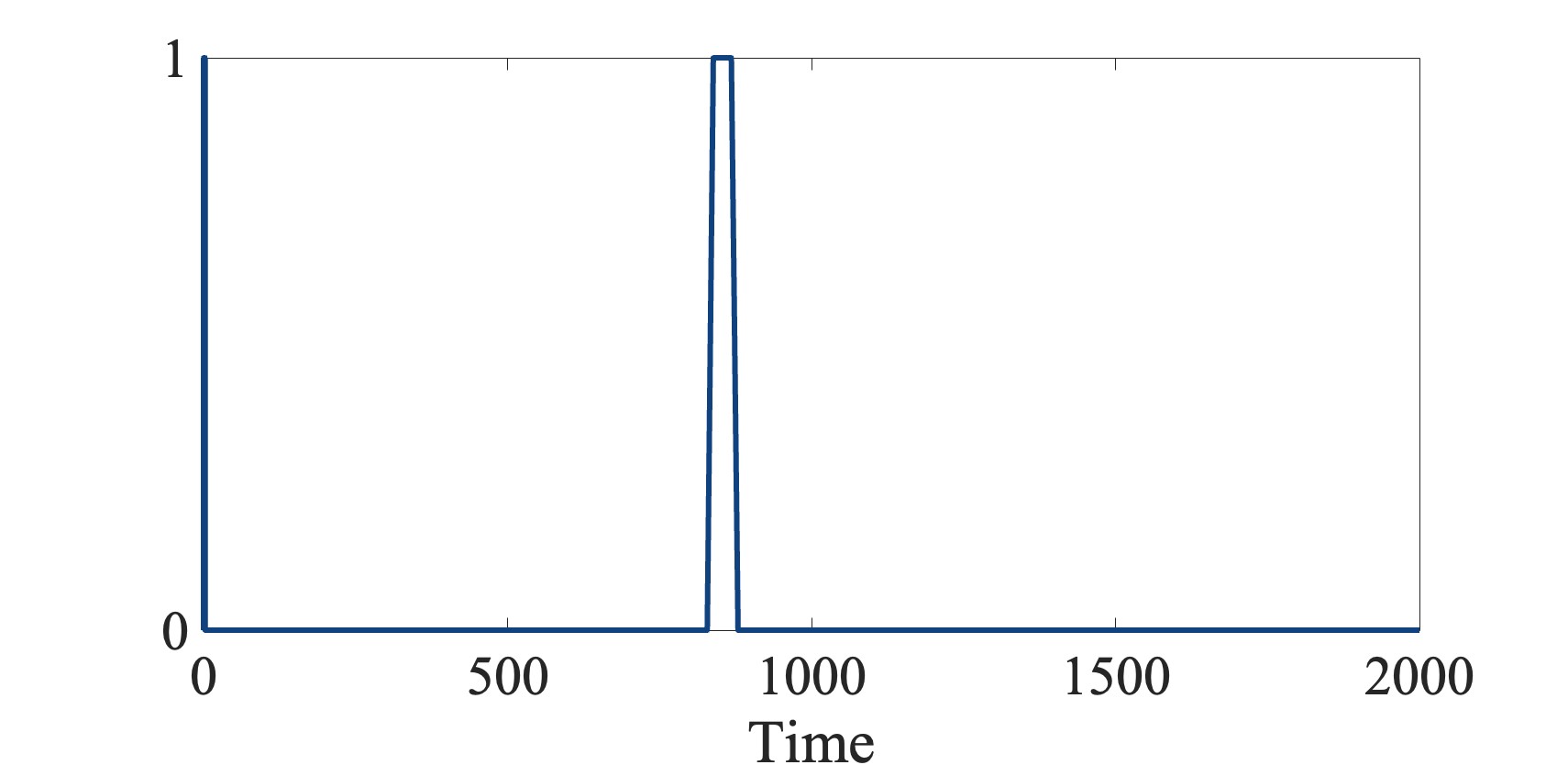}
    \caption{Number of rejected steps per time step}
    \label{fig:3D_PFC_reject_time}
    \end{subfigure}
    \begin{subfigure}[t]{0.49\textwidth}
    \includegraphics[width = \textwidth]{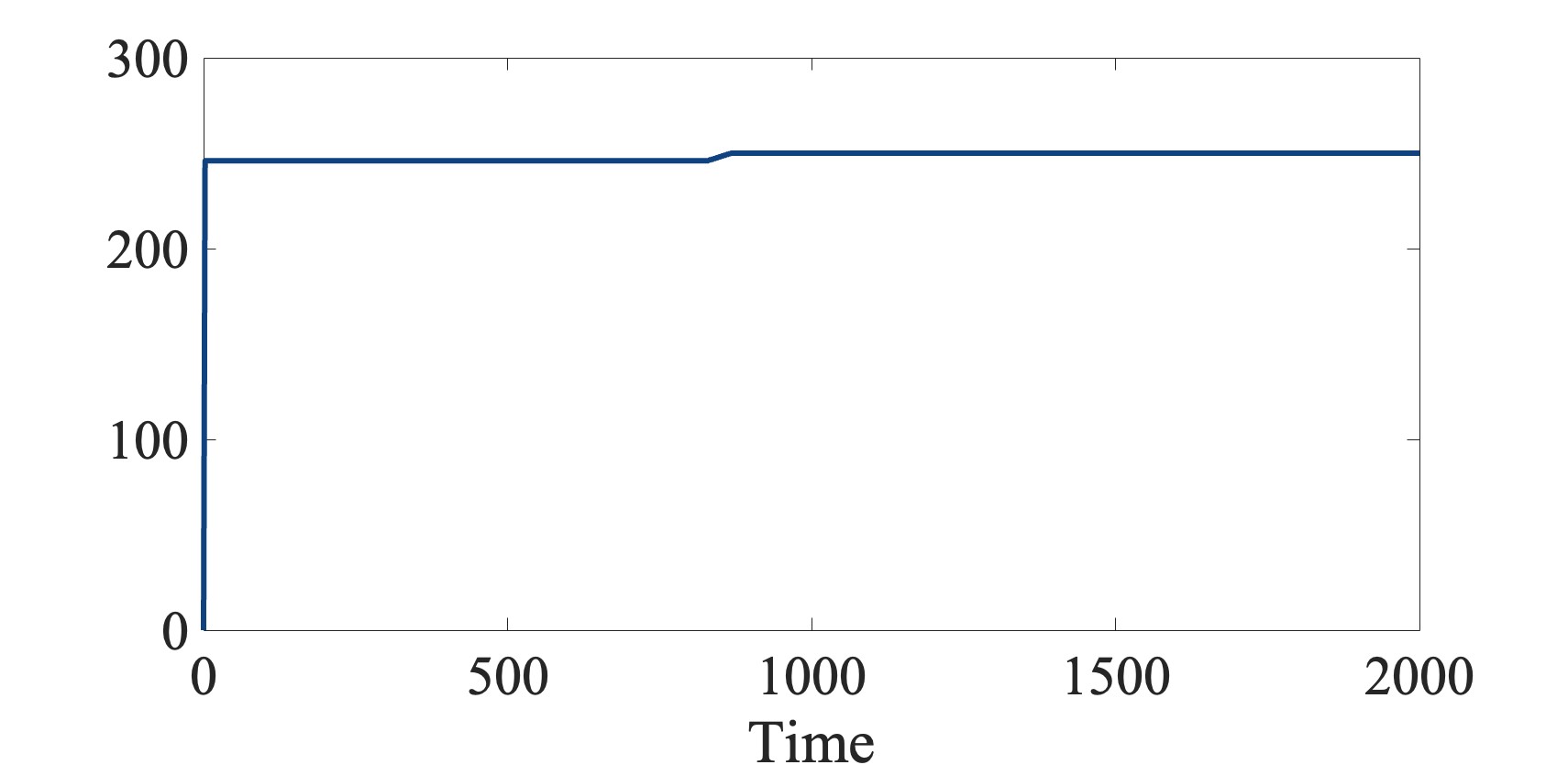}
    \caption{Running total of rejected steps over the time interval}
    \label{fig:3D_PFC_reject_total}
    \end{subfigure}
    \caption{Plots of the number of rejected steps for the 3D PFC model with $m=0.25$.}
    \label{fig:3D_PFC_reject}
\end{figure}

\subsection{Functionalized Cahn--Hilliard}

The FCH model is a variation that includes additional non-linearities and derivative couplings to model the behaviour of lipid bilayer membranes \cite{christlieb2024FCH} which contain both hydrophobic and hydrophilic parts. Consider the following FCH free energy
\begin{equation}
    E(u) = \frac{1}{|\Omega|}\int_{\Omega}\Biggl(\frac{1}{2}(\varepsilon^2\Delta u)^2-\left(1+\frac{\eta}{2}\right)\varepsilon^2|\nabla u|^2\Biggr) \, dx +W(u)\nonumber
\end{equation}
where
\begin{equation}
    W(u)=\frac{1}{|\Omega|}\int_{\Omega}\left(3\varepsilon^2u^2|\nabla u|^2 +\frac{1}{2}(F'(u))^2-\eta F(u)\right) \, dx\nonumber
\end{equation}
and the variational derivative of the nonlinear functional is
\begin{equation}
\frac{\delta G}{\delta u}
=
-6\varepsilon^2\left(u|\nabla u|^2+u^2\Delta u\right)
+F'(u)F''(u)-\eta F'(u).
\nonumber
\end{equation}
Taking the variational derivative gives the following PDE
\begin{equation}
    u_t = \Delta\left(\varepsilon^4\Delta^2 u+\left(2+\eta\right)\varepsilon^2\Delta u +\frac{\delta G}{\delta u}\right)\nonumber
\end{equation}
For the FCH model, we will look at the meandering instability example from \cite{ZhangC2020FCH}. The initial condition is given by a smooth approximation of
\begin{equation}
    \phi_0 = \begin{cases}
        -1,&x>\sin(y)+0.34,\\
        -1,&x<\sin(y)-0.34,\\
        1,&\text{otherwise}
    \end{cases}
\end{equation}
where the process to smooth the initial data is outlined in \cite{ZhangC2020FCH}. Our results are shown in Figure \ref{fig:FCH_wave_IC} and were compared to \cite{ZhangC2020FCH} to verify our results. The profiles of $u$ in Figure \ref{fig:FCH_u} resemble the results in \cite{ZhangC2020FCH}. In Figures \ref{fig:FCH_E} and \ref{fig:FCH_dt}, we see a decrease in the energy for $t\in[0,20]$, so there is a decrease in the time step. For $t>20$ the energy flattens out and the time step starts to increase.

\begin{figure}
\centering
    \begin{subfigure}{0.49\textwidth}
        \includegraphics[width=\textwidth]{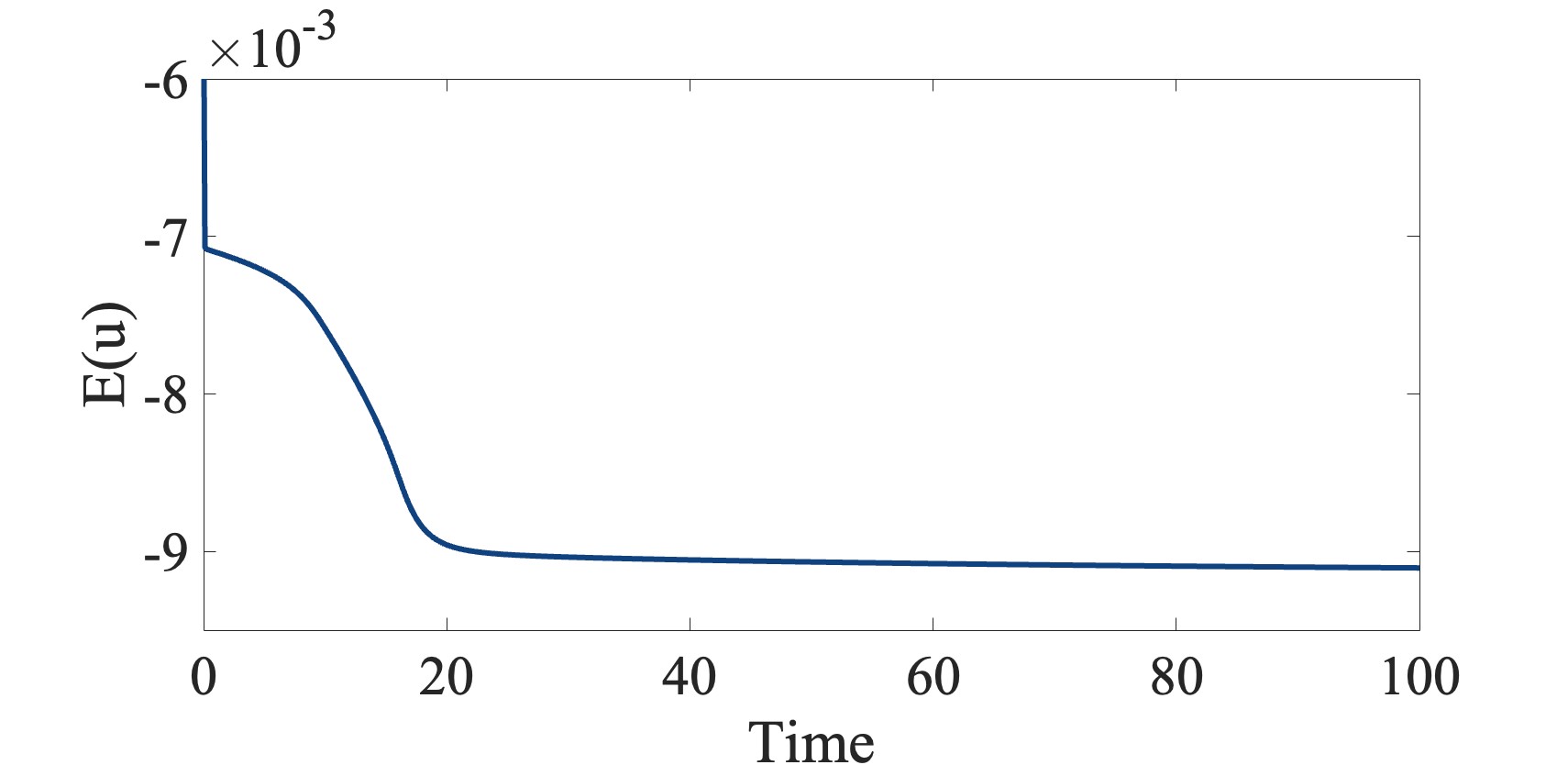}
        \caption{Energy}
        \label{fig:FCH_E}
    \end{subfigure}
    \begin{subfigure}{0.49\textwidth}
        \includegraphics[width=\textwidth]{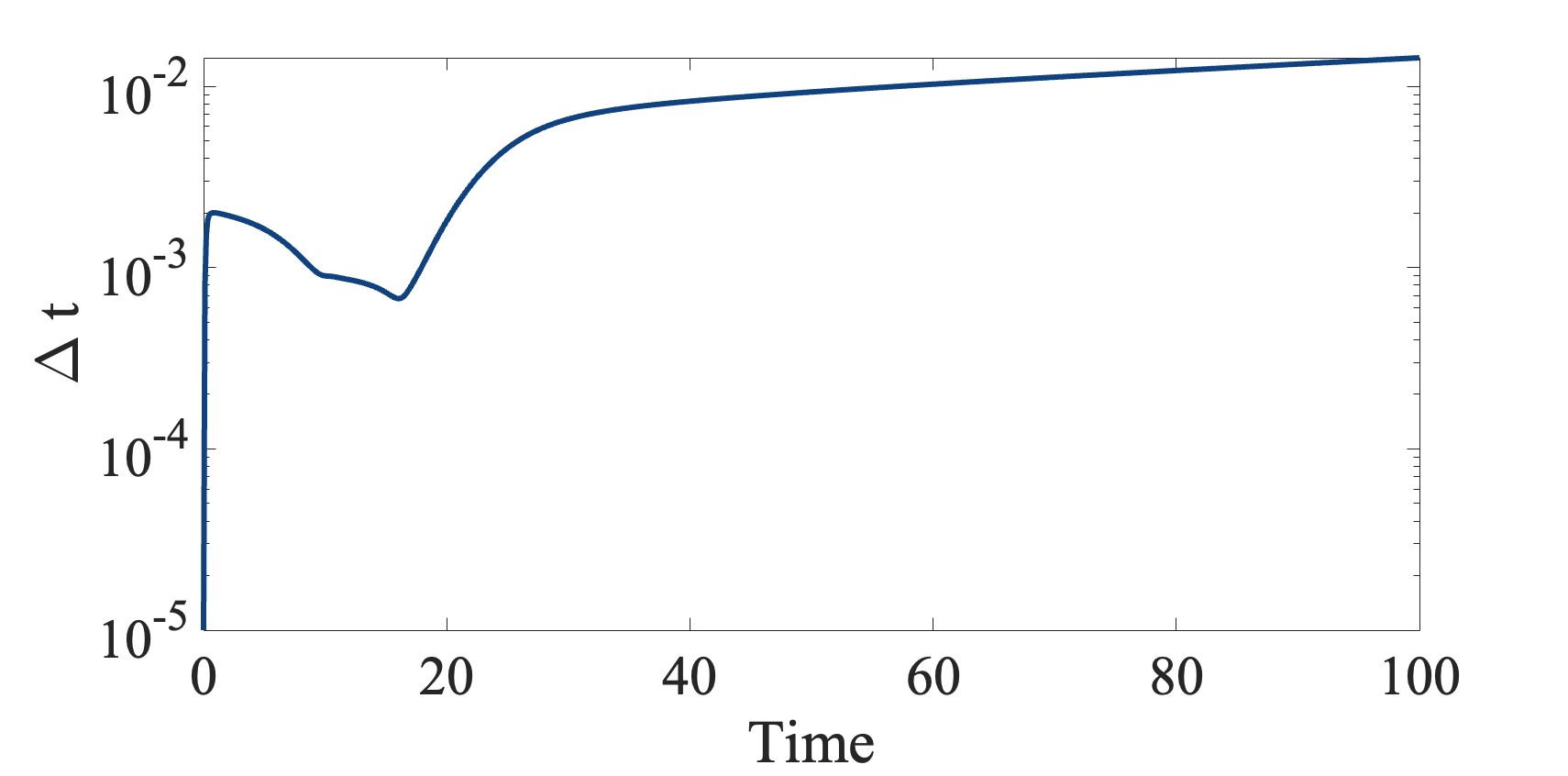}
        \caption{Time step}
        \label{fig:FCH_dt}
    \end{subfigure}

    \begin{subfigure}{0.99\textwidth}
        \includegraphics[width=\textwidth]{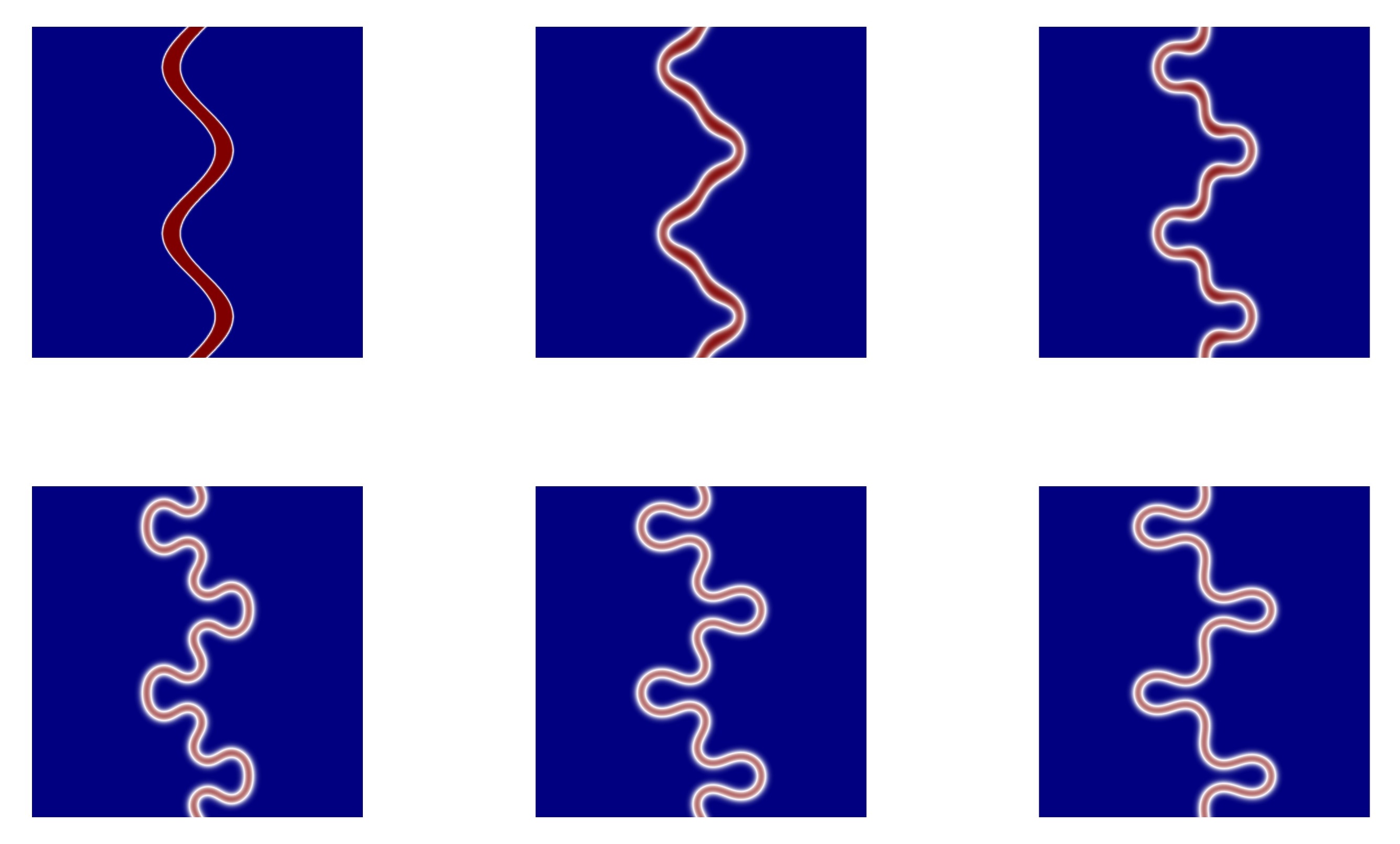}
        \caption{Profiles of $u$ at times $t=0$, 10, 15, 20, 50, and 100.}
        \label{fig:FCH_u}
    \end{subfigure}

    \caption{Plots of energy, time step, and profiles of $u$ at various times for the 2D FCH with smooth initial condition from \cite{ZhangC2020FCH}.}
    \label{fig:FCH_wave_IC}
\end{figure}

\subsection{Efficiency comparison}
\label{sec:efficiency}
We compare our adaptive time stepping scheme with using fixed time step to show how efficient our scheme is. Table \ref{table:num_FFT_full} shows the number of FFT evaluations and the $L_2$ norm of $|\Esym - E(u)|$ for each of the models we've shown. We include runs with various fixed time steps and in Table \ref{table:num_FFT_summary} choose the run with an error closest to the adaptive run. For Cahn--Hilliard and Ohta--Kawasaki we have approximately 10 times the speedup if we use an adaptive method over a fixed time step, for PFC it's approximately 18.9 times and for FCH about 3.5 times.

\begin{table}
  \begin{center}
  \begin{tabular}{|c| c| c |c|} 
  \hline
  \multicolumn{4}{|c|}{CH Spot Collapse}\\
  \hline
   \multicolumn{2}{|c|}{Time step or tolerance} & FFT & $L_2$ Norm of Error  \\
  \hline
  $\tau$ & $10^{-5}$& 84124 & $2.0825\mathrm{e}{-05}$ \\
  \hline
  \multirow{3}{*}{$\dt$} & $10^{-2}$ & 9001& $4.1706\mathrm{e}{-03}$\\
  & $10^{-3}$ & 90001 & $4.9919\mathrm{e}{-05}$\\
  & $10^{-4}$ & 900001 & $1.6478\mathrm{e}{-05}$\\
  \hline
  \end{tabular}
  \end{center}

    \begin{center}
  \begin{tabular}{|c| c| c |c |} 
  \hline
  \multicolumn{4}{|c|}{Ohta-Kawasaki}\\
  \hline
   \multicolumn{2}{|c|}{Time step or tolerance} & FFT & $L_2$ Norm of Error   \\
  \hline
  $\tau$ & $10^{-5}$& 278722 & $6.9410\mathrm{e}{-05}$  \\
  \hline
  \multirow{3}{*}{$\dt$}& $10^{-2}$ & 30001 & $6.0315\mathrm{e}{-04}$  \\
  & $10^{-3}$ & 300001& $1.9743\mathrm{e}{-04}$\\
  & $10^{-4}$ &3000001 & $6.2650\mathrm{e}{-05}$\\
  \hline
  \end{tabular}
  \end{center}

    \begin{center}
  \begin{tabular}{|c| c| c |c |} 
  \hline
  \multicolumn{4}{|c|}{Phase--field crystal}\\
  \hline
   \multicolumn{2}{|c|}{Time step or tolerance} & FFT & $L_2$ Norm of Error  \\
   \hline
  $\tau$ & $10^{-5}$& 3148 & $9.1715\mathrm{e}{-05}$ \\
  \hline
  \multirow{3}{*}{$\dt$}& $10$ & 601 & $2.5389\mathrm{e}{-04}$\\
  & $1$ & 6001& $1.7211\mathrm{e}{-04}$\\
  & $0.1$ &60001 & $8.8917\mathrm{e}{-05}$\\
  \hline
  \end{tabular}
  \end{center}

      \begin{center}
  \begin{tabular}{|c| c| c |c |} 
  \hline
  \multicolumn{4}{|c|}{Functionalized Cahn--Hilliard}\\
  \hline
   \multicolumn{2}{|c|}{Time step or tolerance} & FFT & $L_2$ Norm of Error  \\
  \hline
  $\tau$ & $10^{-3}$& 144141 & $7.4185\mathrm{e}{-06}$ \\
  \hline
   \multirow{3}{*}{$\dt$}& $10^{-3}$ & 500001 & $8.0170\mathrm{e}{-06}$ \\
  & $5\times10^{-4}$ & 1000001
 & $5.6238\mathrm{e}{-06}$\\
 \hline
  \end{tabular}
  \end{center}
  \caption{Number of FFT evaluations and the norm of the error for adaptive and constant time stepping for each example.}
  \label{table:num_FFT_full}
  \end{table}

\begin{table}
  \begin{center}
  \begin{tabular}{lccccccc}
  Model & Fixed $\Delta t$ & Fixed FFTs & Fixed Norm & Adaptive FFTs & Adaptive Norm &Approximate Speedup \\
  \hline
  CH spot & $10^{-3}$ & 900001 &$1.6478\mathrm{e}{-05}$ &84124 &$2.0825\mathrm{e}{-05}$ & 10.7\\
  OK ($m=0.05$)& $10^{-4}$& 3000001& $6.2650\mathrm{e}{-05}$& 278722& $6.9410\mathrm{e}{-05}$ & 10.8\\
  PFC 3D & $10^{-1}$ & 60001& $8.8917\mathrm{e}{-05}$& 3148 &$9.1715\mathrm{e}{-05}$& 18.9\\
  FCH & $10^{-3}$& 500001 & $8.0170\mathrm{e}{-06}$ & 115312 & $7.4185\mathrm{e}{-06}$& 3.5
  \end{tabular}
  \end{center}
  \caption{Condensed version of Table \ref{table:num_FFT_full} presenting the adaptive time-step run alongside the fixed time-step run whose error norm most closely matches the adaptive case. This table also reports the approximate speedup.}
  \label{table:num_FFT_summary}
  \end{table}

\section{Conclusions}
\label{sec:conclusions}

We have shown that energy-rate--based adaptive time stepping for gradient flows has a natural interpretation as equidistribution of a monitor function along the energy curve. The monitor $g = |\dot{E}|$ is asymptotically optimal in the minimax sense: it minimises the leading-order maximum symmetric-energy discrepancy for a fixed step budget, and the optimal leading-order value is $(E(0)-E(T))/(2N)$, depending only on the total energy drop and the number of steps. A combined monitor that switches to $g = \sqrt{|\dot{E}|}$ near equilibrium caps step growth where the first rule becomes vacuous, and does so without weakening the energy-error bound: the combined rule controls $\dt_n|\dot E_n|$ by $\tau_1$ in both regimes. The arclength parameter $\beta$ used in prior adaptive stepping schemes \cite{Qiao2011,Zhang2012} is shown as $\beta = \Delta t_{\max}/\tau_1$, the ratio of the largest permitted step to the energy-error tolerance. These results connect the ideas energy-stable phase-field numerics and mesh equidistribution which developed largely independently.

The adaptive controller built on this framework proposes steps from the energy dissipation rate and accepts them only when a computable symmetric energy $\Esym^n$ does not increase. Because the accepted $\Esym$ values form a monotone chain, this yields a bound on the physical energy that is uniform in the number of steps, $\sup_n E(u^n) \le E(u^0) + \tau_1 + O(\Delta t_{\max}^2)$: the user's tolerance caps the admissible energy rise over the entire integration, not merely per step. Per step, the deviation from monotonicity is proportional to $(1-\rho_n^{-1})_+$, so it is a step-\emph{ratio} rather than a step-\emph{size} effect and vanishes when the controller keeps ratios near unity. The SSAV--vBDF2 scheme achieves second-order global accuracy under standard step-ratio restrictions.  The guarantee is a posteriori, depending on a rejection mechanism. Across Cahn--Hilliard, Ohta--Kawasaki, PFC, and FCH test problems, the controller takes small steps during rapid energy changes and grows the step aggressively on plateaus, with no accepted step violating energy monotonicity.

Several extensions are worth pursuing. Higher-order BDF methods (BDF3--BDF6) can reach up to sixth-order accuracy, and the equidistribution analysis would need to account for the higher-order local truncation error structure. Whether the energy-error bounds remain useful as $T \to \infty$ depends on the asymptotic dynamics of the specific gradient flow; for flows approaching stable equilibria one expects the situation to improve as $|\nabla E| \to 0$, but a rigorous statement would require long-time stability estimates. Vector-valued problems such as Cahn--Hilliard--Navier--Stokes can in principle be handled the same way, though the energy structure is more involved and the interaction between different dissipation mechanisms (viscous and diffusive) may require separate monitors or a multi-component equidistribution strategy.

\section*{Data availability}
No data was used for the research in this article.

\section*{Acknowledgements}
The authors thank NSERC (2021-04114) for financial support and John Stockie for helpful discussions. 

\subsection*{Declaration of generative AI}
During the preparation of this manuscript the authors used Writefull for Overleaf in order to paraphrase some sentences and make edits to the document. The authors also used Claude code to comment and debug some of the code. The authors reviewed and edited the output and take full responsibility for the content of the publication.

\bibliographystyle{plain}
\bibliography{references}
\end{document}